\documentclass[a4paper]{amsart}

\usepackage{amsthm, amsfonts, amssymb, amsmath, latexsym, enumitem,array}
\usepackage[all]{xy}
\usepackage{hyperref,cite,mdwlist}
\usepackage{mathrsfs}
\usepackage{tikz}
\usepackage{libertine}
\usepackage[libertine,scaled=.9]{newtxmath}
\usepackage{dsfont}
\usepackage {color}

\newtheorem{thm}{Theorem}[section]

\newtheorem{lem}[thm]{Lemma}
\newtheorem{corol}[thm]{Corollary}
\newtheorem*{corol*}{Corollary}
\newtheorem{prop}[thm]{Proposition}

\newtheorem*{thm*}{Theorem}
\newtheorem*{cnj*}{Conjecture}
\newtheorem{claim}[thm]{Claim}

\theoremstyle{definition}
\newtheorem{rmk}[thm]{Remark}

\newtheorem{dfn}[thm]{Definition}

\newtheorem*{conj*}{Conjecture}

\newcommand{\cA}{\mathcal{A}}
\newcommand{\cB}{\mathcal{B}}

\newcommand{\cE}{\mathcal{E}}

\newcommand{\cL}{\mathcal{L}}

\newcommand{\cF}{\mathcal{F}}

\newcommand{\cG}{\mathcal{G}}

\newcommand{\cI}{\mathcal{I}}

\newcommand{\cO}{\mathcal{O}}

\DeclareMathOperator{\Ext}{Ext}
\DeclareMathOperator{\Pic}{Pic}

\DeclareMathOperator{\Hom}{Hom}

\DeclareMathOperator{\HH}{H}

\DeclareMathOperator{\codim}{codim}
\DeclareMathOperator{\di}{dim}

\newcommand{\ZZ}{\mathbb Z}
\newcommand{\CC}{\mathbb C}

\newcommand{\FF}{\mathbb F}

\newcommand{\PP}{\mathbb P}
\newcommand{\GG}{\mathbb G}

\newcommand{\bk}{{\boldsymbol k}}

\newcommand{\ra}{\rightarrow}

\numberwithin{equation}{section}

\allowdisplaybreaks[1]
\begin{document}


\title{Ulrich bundles on three dimensional scrolls}

\author{Maria Lucia Fania}
\email{marialucia.fania@univaq.it}
\address{DISIM
University of L'Aquila
Via Vetoio, Loc. Coppito
I-67100 L'Aquila, Italy}

\author{Margherita Lelli-Chiesa}
\email{margherita.lellichiesa@univaq.it}
\address{DISIM
University of L'Aquila
Via Vetoio, Loc. Coppito
I-67100 L'Aquila, Italy}

\author{Joan Pons-Llopis}
\email{juanfrancisco.ponsllopis@univaq.it}
\address{DISIM
University of L'Aquila
Via Vetoio, Loc. Coppito
I-67100 L'Aquila, Italy}



\thanks{The first and second authors were partially supported by the Italian PRIN-2015 project \lq\lq\thinspace\!Geometry of Algebraic varieties".}

\begin{abstract}
In this paper we construct Ulrich bundles of low rank on three-dimensional scrolls (with respect to the tautological line bundle). We pay special attention to the four types of threefold scrolls in $\PP^5$ which were classified in \cite{ottaviani:scrolls}.
\end{abstract}

\maketitle

%
%

\section*{Introduction}

Ulrich bundles on projective varieties were introduced (under an algebraic disguise) more than thirty years ago in \cite{ulrich:high-numbers}. At that time, they were called in a different way: linear maximal, or maximally generated, Cohen-Macaulay modules. However, it is with the publication of \cite{eisenbud-schreyer-weyman} that Ulrich bundles stepped into the realm of Algebraic Geometry fully fledged. In this groundbreaking paper, the authors showed that Ulrich bundles could become a crucial object to tackle a wide range of problems. For instance, one associates with a projective variety $X\subset\PP^n$ of dimension $k$ the so called \emph{Chow divisor} $D_X$ of $X$ in the Grassmanian $Gr(k+1,n)$ parametrizing the $(k+1)$-codimensional planes that intersect $X$. The Chow divisor is  classically  studied and it is a challenging problem to give an explicit presentation of its defining equation, the so called \emph{Chow form} of $X$.
In \cite{eisenbud-schreyer-weyman}, it was showed that the Chow form of $X$ is given by the determinant of a square matrix with linear entries (in the Pl\"{u}cker variables) whenever $X$ supports an Ulrich bundle. All these reasons motivate the question whether any projective variety supports Ulrich bundles.

A vector bundle $\cE$ on a $k$-dimensional projective variety $X\subset\PP^n$ is called Ulrich if it has an $\cO_{\PP^n}$ linear resolution of length $\codim_{\PP^n}(X)=n-k$:
$$
0\rightarrow \cO_{\PP^n}(k-n)^{a_{n-k}}\rightarrow\cO_{\PP^n}(k-n+1)^{a_{n-k-1}}\rightarrow\cdots\cO_{\PP^n}(-1)^{a_1}\rightarrow\cO_{\PP^n}^{a_0}\rightarrow\cE\rightarrow 0,
$$
\noindent that is, all the morphisms of the resolution are given by matrices of linear forms in the polynomial ring $\bk[x_0,\cdots, x_n]$. Among many other definitions, being Ulrich is equivalently stated in terms of the cohomology vanishings $H^{\bullet}(\cE(-i))=0$ for $i=1,\dots, \di(X)$.

The existence of Ulrich bundles on arbitrary projective varieties has interested many ma\-the\-maticians but is still unknown in full generality. Until now, it has been proved for curves, del Pezzo surfaces, rational ACM surfaces in $\PP^4$, ruled surfaces, Abelian surfaces, surfaces with $q=p_g=0$, K3 surfaces, Fano threefolds of even index,  Grassmannians, Segre varieties, certain flag manifolds, ... (see \cite{casanellas-hartshorne-geiss-schreyer},\cite{miro_roig-pons_llopis:surfaces_P4}, \cite{aprodu-costa-miro},\cite{beauv-abelian-ulrich},\cite{casnati-ulrich},\cite{aprodu-farkas-ortega},\cite{beauville-ulrich-intro}, \cite{costa-miro_roig-pons_llopis}, \cite{miro-roig-pons-llopis:fano}, \cite{coskun-costa-hui-miro}, \cite{costa-miro-ulrich-grass}).

If a variety $X$ supports Ulrich bundles, another interesting problem is to compare their lowest rank with respect to to the dimension of $X$; this is generally large. For instance, it has been conjectured in \cite[Conjecture B]{buchweitz-greuel-schreyer} that any Ulrich vector bundle on an smooth hypersurface $X\subset\PP^n$ (distinct from the hyperplane $\PP^{n-1}$) should have rank greater or equal than $2^e$ with $e=[\frac{n-2}{2}]$.

The goal of this paper is to approach the existence of Ulrich bundles of low rank on a wide class of
projective threefolds, namely, projective scrolls of dimension $3$. Because  the existence of Ulrich bundles  on geometrically ruled surfaces has  been considered in  \cite{aprodu-costa-miro}, it is natural to investigate what happens on a $3$-dimensional projective scroll.

It was proved by Ottaviani in \cite{ottaviani:scrolls} that there are only four types of smooth $3$-dimensional scrolls over a surface that are embedded in  $\PP^5$, all classical,  namely,  the Segre scroll, the Bordiga scroll, the Palatini scroll and the scroll over a $K3$ surface; they have
degree $3, 6, 7, 9$, respectively. Note that  the base surface of both the Segre scroll and  the Bordiga scroll  is $\PP^2$, while in the case of the Palatini scroll the base is a smooth cubic surface in $\PP^3$.

 In codimension greater than $2$, from the list of known threefolds of low degree  we see that  the base surface  of the scroll is either $\PP^2$,  or a smooth quadric surface $\mathcal Q \subset \PP^3$,   or $\FF_1$ (cf., e.g.,  \cite{fania-livorni-deg10},  \cite[Table 1, Table 2]{besana-fania},  \cite{besana-biancofiore}, \cite[Remark 3.3 and \S 7]{alzati-besana}).

 For this class of $3$-dimensional smooth varieties we investigate the existence of Ulrich bundles of  rank one and two.

 We gather our results concerning the above scrolls in the following:

 \begin{thm}\label{main results}
Let $(X,L) = (\mathbb{P}(\cE), \xi)$ be a $3$-dimensional scroll over a smooth surface $S$. Let  $\xi={\mathcal O}_{\mathbb{P}(\cE)}(1)$ be the tautological line bundle and $\pi: \mathbb{P}(\cE) \to S$ be the projection morphism. Let $X$ be embedded by $|L|=|\xi|$ in $\PP^N$. Then, the following hold:
\begin{enumerate}
\item Assume that $S$ is a $K3$ surface that is a transversal linear section of the Grassmannian $\mathbb{G}(1,5)\subset \PP^{14}$. If $Pic(S)=\ZZ[H]$, then
 $X$ does not support  any Ulrich line bundle but it  supports stable  Ulrich bundles $\cG$ of rank two (Theorem \ref{Ulrich on general K3scroll}).

\noindent If instead  $Pic(S)=\ZZ[H]\oplus \ZZ[C]$ with $H^2=C^2=14$ and $H\cdot C=16$, then $X$ supports both Ulrich line bundles and stable rank two Ulrich bundles (Proposition \ref{Ulrich on K3scroll} and Theorem \ref{stable rk2 Ulrich on K3scroll}).

  \item  If $X$ is a Palatini scroll, up to permutation of the exceptional divisors of the cubic surface $S\subset \PP^3$, there are exactly six Ulrich line bundles on $X$ (Proposition \ref{Ulrich on Palatini}). Furthermore, $X$ supports stable  rank two Ulrich bundles (Theorem \ref{stable rk2 Ulrich on PalatiniScroll}).
\item  If $X$ is a scroll over $\PP^2$, then it supports some Ulrich line bundles if and only if $(c_1(\cE),c_2(\cE))\in\{(2,1),(4,10),(4,6),(5,15)     \}$; in all of these cases, there are exactly two  Ulrich  line bundles (Propostion  \ref{p_2}). On the other hand, $X$ always supports stable rank $2$ bundles (Theorem \ref{stable rk2 Ulrich on Scrolls overP2}).
\item If $X$ is a scroll over ${\mathcal Q}^2$, then it does not support any Ulrich line bundle unless  $c_1({\mathcal O}_{{\mathcal Q}^2}(3,3)$ and $c_2(\cE)=9$, in which case there are exactly two Ulrich  line bundles (Propostion  \ref{Q_2_and_F1}). Furthermore, in this case  $X$ also supports stable rank $2$ bundles (Theorem \ref{stablerank2UlrichonQ2}).
\item If $X$ is a scroll over $\FF_1$, then it does not support any Ulrich line bundle (Propostion  \ref{Q_2_and_F1}); however, it supports stable rank two Ulrich bundles  (Theorem  \ref{rk2Ulrich on scroll over F1}).
\end{enumerate}
\end{thm}

Also note that in the cases of the above Theorem where Ulrich line bundles on $X$ do exist, these are completely classified.

The structure of the paper is as follows. In Section \ref{Notation and Preliminaries} we collect all the necessary notation and background material.  In Section \ref{generalities} we prove some general results about Ulrich bundles  on projective scrolls. Let $\pi:X=\PP(\cE)\to Y$ be a projective bundle and denote by $\xi={\mathcal O}_{\mathbb{P}(\cE)}(1)$  the tautological line bundle. Theorem \ref{equivalent conditions} provides numerical criteria  for a line bundle $a\xi+\pi^*D$ with $a\in\ZZ$ and $D\in\Pic(Y)$ to be Ulrich, in terms of vanishing of the cohomology groups for vector bundles of the form $S^j\cE(D)$ on $X$.  In particular, we recover both \cite[Prop. 5]{beauville-ulrich-intro} and  \cite[Theo. 2.1]{aprodu-costa-miro}. Theorem \ref{pullback} gives a method in order to construct Ulrich bundles on $X$ twisting by $\xi$ the pullback of some vector bundles on the base $Y$ that satisfy certain vanishings.
In Section \ref{Scroll over a $K3$ surface} we consider $3$-dimensional scrolls $X\subset \PP^5$ over $K3$ surfaces $S$  that are transversal intersections of the $8$-dimensional Grassmannian $\mathbb{G}(1, 5) \subset \PP^{14}$ with a $\mathbb{P}^8\subset \PP^{14}$. We will prove that  if \, $\mathrm{Pic}(S)$ has rank $1$ then the  $3$-dimensional scroll  $X$ does not support any  Ulrich line bundle, while  it carries stable Ulrich bundles of rank $2$. The existence of Ulrich line bundles is then established when $S$ lies in a Noether-Lefschetz divisor as in the statement of Theorem \ref{main results} (cf. Proposition \ref{Ulrich on K3scroll}).

In Section \ref{Palatini Scroll}  we show that on the Palatini  $3$-fold $X\subset\PP^5$ there are three Ulrich line bundles of type $L=2\xi+\pi^{*}D$ for appropriate $D$, jointly with their respective companions, cf. Proposition \ref{Ulrich on Palatini}. Such Ulrich line bundles are then used to construct rank two Ulrich bundles which turn out to be simple.
The constructed rank two Ulrich bundles, being extensions of line bundles, cannot be stable.  Theorem \ref{stable rk2 Ulrich on PalatiniScroll} takes care of the existence of  stable rank two Ulrich bundles on the Palatini scroll.

In Section \ref{Scroll over P2;Q2;F1} we prove the existence of Ulrich bundles of rank $1$ and $2$ on  $3$-dimensional scrolls $X$ over $\PP^2$ and ${\mathcal Q}^2$. Furthermore, we show that $3$-dimensional scrolls over $\FF_1$ do not support Ulrich line bundles, while they carry stable rank two Ulrich bundles.  In some of the cases listed in  Proposition \ref{p_2}, the existence of Ulrich  bundles was already known using a different approach.

Finally, in Section \ref{ultima} we perform an opposite construction to that of Theorem \ref{pullback} in the case of $3$-dimensional scrolls $X$ over a surface $S$. More precisely, we prove that (under some splitting type hypotheses) a Ulrich bundle $\cG$ on $X$ (with respect to $\xi$) can be twisted appropriately so that the pushforward of the resulting sheaf is a Ulrich bundle on $S$ with respect to $c_1(\cE)$ having the same rank as $\cG$. This along with Theorem \ref{pullback} provides a one to one correspondence:
$$
\left\{ \begin{array}{c}
          \text{Ulrich bundles $\cF$ of rank $r$ on $S$ } \\
          \text{with respect to $c_1(\cE)$}
        \end{array}\right\}_{\biggm/ \cong_{iso}}\Leftrightarrow\left\{\begin{array}{c}
          \text{Ulrich bundles $\cG$ of rank $r$ on $X$} \\
          \text{ with respect to $\xi$ such that} \\
           \text{$\cG_{|\pi^{-1}(s)}\cong\cO_{\PP^1}(1)^r$ for $s\in S$}
        \end{array}\right\}_{\biggm/ \cong_{iso}}.
$$

%
%

\section{Notation and Preliminaries}
\label{Notation and Preliminaries}
We work over the complex numbers $\CC$. By a variety $X$ we mean an irriducible and reduced projective scheme. Cartier divisors, their associated line bundles, and the invertible sheaves of their holomorphic sections are used with no distinction. Mostly additive notation is used for their group.
Given a sheaf $\cF$ on $X$ with $S^{j}(\cF)$ we denote the $j$-th symmetric product of $\cF$ and with $\cF(D)$ the twist of $\cF$ by a divisor $D$.

 For the reader convenience we recall some well known facts that we will use in the sequel.

 \begin{dfn}
A smooth $3$-dimensional variety  $X\subset \PP^N$ is said to be a scroll over a smooth surface $S$ if $X=\mathbb{P}(\cE)\xrightarrow{\pi} S$, where $\cE$ is a rank two vector bundle over $S$ and   the embedding in  $ \PP^N$ is given by the line bundle  $\cO_{\mathbb{P}(\cE)}(1)$.
 \end{dfn}
\begin{prop}
\label{properties-taut-bundle}
Let $Y$ be a polarized manifold of dimension $m$  and let $\cE$ be a rank $r+1$ vector bundle on $Y$.
 Let $X=\mathbb{P}(\cE)$,  let $\xi= \cO_{\mathbb{P}(\cE)}(1)$ be the tautological line bundle  and let  $\pi: X=\mathbb{P}(\cE) \rightarrow Y$ be the  bundle  projection. Then
 $$ Pic(X) \cong \ZZ\oplus \pi^{*}Pic Y.$$
Moreover we have
\begin{itemize}
\item[(i)] $\pi_{*} \cO_X(l) \cong S^{l}\cE$ for $l\ge 0$,  $\pi_{*} \cO_X(l) = 0$ for $l < 0$,
\item[(ii)] $R^{i} \pi_{*} \cO_X(l) = 0$ for $0<i<r$ and all $l\in \ZZ$ and $R^{r} \pi_{*} \cO_X(l) = 0$ for $l > -r-1$,
\item[(iii)] For any $l\in \ZZ$
$$R^{r} \pi_{*} \cO_X(l) \cong \pi_{*} \cO_X(-l-r-1)^{\vee}\otimes c_1(\cE)^{\vee} $$
\end{itemize}
\end{prop}
For details see \cite[Exercise 8.4, pg 252]{hartshorne:ag}.
 \begin{dfn}
Let $X\subset \PP^N$ be a smooth variety of dimension $n$  polarized by $H$ where $H$ is a hyperplane section of $X$.
A vector bundle $\cF$  on $X$ is said to be  {\it Ulrich} with respect to $H$ if
$$H^{i}(X, \cF(-jH))=0 \quad \text{for}  \,\, i=0, \cdots, n  \,\, \text{and} \,\, 1 \le j \le \di X. $$
 \end{dfn}

In the following Proposition we gather some of the properties of an Ulrich bundle $\cF$ that will be used throughout the paper,  see \cite{casanellas-hartshorne-geiss-schreyer}:

\begin{prop}\label{Ulrichproperties}
 Let $X\subset \PP^N$ be a smooth variety of dimension $n$  polarized by $H$ and let $\cF$ be an Ulrich bundle on $X$ with respect to $\cO_X(H)$. Then:
\begin{itemize}
\item[(i)] The restriction $\cF_{H}$  to a general hyperplane section $H$ of $X$ is again an Ulrich bundle.
\item[(ii)]  $h^{0}(X,\cF) =  rk(\cF)\cdot deg(X)$.
\item[(iii)] Ulrich bundles are $\mu$-semistable (equivalently semistable) with respect to the polarization $\cO_X (H)$. Moreover, if $\cF$ is strictly semistable, there exists an exact sequence of vector bundles:
    $$
    0\rightarrow\cE\rightarrow\cF\rightarrow\cG\rightarrow 0,
    $$
\noindent with $\cE$ and $\cG$ Ulrich bundles of lower rank.
\end{itemize}
\end{prop}

 \begin{rmk}
If $\cF_1$ is a vector bundle on $X$ which is Ulrich with respect to $H$  then
$\cF_2=\cF_1^{\vee}(K_X +(n+1)H)$ is also Ulrich with respect to $H$.  Indeed, we have
$$H^{i}(X,\cF_2(-jH))=H^{i}(X,\cF_1^{\vee}(K_X +(n+1-j)H))\cong H^{n-i}(\cF_1(-(n+1-j)H)=0 \,\, \text{for}  \,\,  i \le n  \,\,  \text{and} \,\,  1 \le j' \le \di X$$
where $j'=n+1-j$.

 From this we see that  Ulrich  bundles come in pairs.
\end{rmk}
 \begin{dfn} Let $X\subset \PP^N$ be a smooth variety of dimension $n$  polarized by $H$, where $H$ is a hyperplane section of $X$,    and let $\cF$  be  a rank $2$ Ulrich bundle  on  $X$. Then $\cF$ is said to be {\it special} if $c_1(\cF) = K_{X} +(n + 1)H.$
 \end{dfn}
It is interesting to notice that in the case of a surface $S$ the fact that a vector bundle $\cF$ is special Ulrich depends basically on its Chern classes. More specifically it holds (see \cite[Corollary 2.4]{casnati-ulrich}):
\begin{prop}\label{ulrichchernclassessurface}
  Let $(S,H)$ be a polarized surface. If $\cF$ is a vector bundle of rank $2$ on S, then the following assertions are equivalent:
\begin{itemize}
  \item $\cF$ is a special Ulrich bundle;
  \item $\cF$ is initialized (that is, $H^0(S, \cF(-H))=0$ and $H^0(S,\cF)\neq 0$) and
 \begin{equation}\label{chernclasses}
   c_1(\cF) = 3H + K_S, \quad\text{and}\quad c_2(\cF) =\frac{1}{2}(5H^2 + 3H K_S) + 2\chi(\cO_S).
 \end{equation}
\end{itemize}
\end{prop}

\vspace{3mm}
We recall the definition of Lazarsfeld-Mukai bundles associated with a complete, base point free linear series on a smooth irreducible curve $C$ lying on a $K3$ surface.
 \begin{dfn}
Let $S$ be a $K3$ surface and  $C \subset  S$ be a smooth irreducible curve. If $A$ is a complete, base point free linear series of type $g^r_d $ on $C$, the Lazarsfeld-Mukai bundle $E_{C,A}$ associated with the pair $(C,A)$ is defined as the dual of the kernel of the evaluation map $ev: H^{0}(C,A)\otimes  \cO_{S}\twoheadrightarrow A$. Hence, $E_{C,A}$ sits in the short exact sequence
$$0\to H^{0}(C,A)^{\vee} \otimes \cO_{S} \to E_{C,A}  \to \omega_{C}\otimes A^{\vee} \to 0.$$
 \end{dfn}

%
%

\section{Generalities}
\label{generalities}

In this section we will prove some theorems about Ulrich vector bundles on projective bundles. We state them in a more general setting than what we need for our purpose.
\begin{thm}
\label{equivalent conditions}
  Let $(Y,H)$ be a polarized manifold of dimension $m$ with $H$ very ample and let $\cE$ be a rank $r+1$ vector bundle on $Y$ such that $\cE$ is (very) ample and spanned. Let $\cL:=a\xi+ \pi^{*}(D)$ be a Ulrich line bundle on the projective bundle $X:=\mathbb{P}(\cE) \xrightarrow{\pi} S$ for some $a\in\mathbb{Z}$ and $D\in Pic(Y)$; then, $a=0,1,\dots.. m$.  Moreover:
  \begin{itemize}
    \item[(i)] A line bundle $\cL_1=a\xi+\pi^{*}(D)\in Pic(X)$  is Ulrich (with respect to $\xi$) if and only if  $\cL_2:=(m-a)\xi+\pi^{*}(c_1(\cE)+K_Y-D)$ is Ulrich. In particular we only need to study line bundles $\cL=a\xi+\pi^{*}(D)$, with $D\in Pic(Y)$ and $\ulcorner\frac{m}{2}\urcorner\le a\leq m$.
    \item[(ii)] A line bundle $\cL=a\xi+\pi^{*}(D)$ with  $\ulcorner\frac{m}{2}\urcorner\le a\leq m$ is Ulrich (with respect to $\xi$) if and only if the following conditions hold:
    \begin{enumerate}
      \item[$(\alpha)$] $H^i(Y,S^{j}\cE(D))=0$ for $j=0,\dots a-1$ and $i=0,\dots m$.
      \item[$(\beta)$] $H^i(Y,S^{j}\cE^\vee(D-c_1(\cE))=0$ for $j=0, \dots m-a-1$ and $i=0,\dots m$.
    \end{enumerate}
  \end{itemize}
\end{thm}
\begin{proof}
If a line bundle $\cL_1:=a\xi+ \pi^{*}(D)$ with $D\in Pic(Y)$  is Ulrich with respect to $\xi$  then $a \ge 0$ since $h^0(\cL_1)\neq 0$. Its partner $\cL_2=K_X +(dim X+1)\xi-\cL_1= (m-a)\xi +\pi^{*}(c_1(\cE)+K_Y-D)$ is also Ulrich with respect to $\xi$   and thus $m-a \ge 0$.

Item (i) is trivial.

Consider a line bundle $\cL=a\xi+ \pi^{*}(D)$ with $\ulcorner\frac{m}{2}\urcorner\le a\leq m$ and $D\in Pic(Y)$; we need to compute $H^{i}(\cL (-t\xi)= H^{i}((a-t)\xi + \pi^{*}(D))$ for $t=1, \cdots, dim X=m+r$ and for $i=0,\dots m+r$.

 If $0\le j:=a-t\le a-1$, then $H^{i}(X,(a-t)\xi + \pi^{*}(D))\cong H^{i}(Y, S^{j}\cE(D))$ for all $i$.

 If $-r\le j=a-t\le -1$, then $H^{i}(X,(a-t)\xi + \pi^{*}(D))=0$ since $\pi_{*}(a-t)\xi$ is the zero sheaf.

If $a-m-r\le j=a-t\le -r-1$ then $H^{i}(X,(a-t)\xi + \pi^{*}(D))\cong H^{m+r-i}(X,(t-a-r-1)\xi + \pi^{*}(c_1(\cE+K_Y-D))\cong H^{m+r-i}(Y, S^{-r-1-j}\cE(K_Y+c_1(\cE)-D))\simeq H^{i-r}(Y,S^{-r-1-j}\cE^\vee(D-c_1(\cE)))$ for all $i$.

 Thus $\cL=a\xi+ \pi^{*}(D)$ is Ulrich with respect to $\xi$  if and only if conditions $(a)$ and $(b)$ hold.
\end{proof}

Whenever it is clear from the contest,  in the notation  $\cL=a\xi+\pi^{*}(D)$ we will drop $\pi^{*}$ and simply write  $\cL=a\xi+D$.

The following immediate corollary holds for threefolds.
\begin{corol}
\label{equivalent conditions for line bdl}
  Let $(S,H)$ be a smooth  polarized surface with $H$ very ample and let $\cE$ be a rank two vector bundle on $S$ such that $\cE$ is (very) ample and spanned. If a line bundle $\cL_1:=a\xi+ D$ on the projective bundle $X:=\mathbb{P}(\cE)\xrightarrow{\pi} S$ is Ulrich, then $a=0,1,2$. Moreover:
  \begin{enumerate}
    \item A line bundle of the form $\cL_1:=\xi+D$ with $D\in Pic(S)$ is Ulrich (with respect to $\xi$) if and only if $H^i(S,D)=H^i(S,D-c_1(\cE))=0$ for $i=0,1,2$.
    \item A line bundle of the form $\cL_1:=2\xi+D$ with $D\in Pic(S)$ is Ulrich (with respect to $\xi$) if and only if $H^i(S,D)=H^i(S,\cE(D))=0$ for $i=0,1,2$.
    \item A line bundle of the form $\cL_1:=D$ with $D\in Pic(S)$ is Ulrich (with respect to $\xi$) if and only $\cL_2:=2\xi+c_1(\cE)+K_S-D$ is Ulrich.
      \end{enumerate}
\end{corol}
\begin{proof}
Just apply Theorem \ref{equivalent conditions}.
\end{proof}

Line bundles as in $(1)$ (respectively, $(2)$) are said to be of type $(1)$  (respectively, type $(2)$).
\vspace{3mm}

\begin{rmk}
 For $m=1$, we recover \cite[Prop. 5]{beauville-ulrich-intro}. Moreover, notice that on the case of geometrically ruled surfaces, if $\cL=\xi+\pi^{*}(D)$ is Ulrich, we know  that the restriction of $\cL$ to a curve $C\in|\xi|$ is also Ulrich, see Proposition \ref{Ulrichproperties},  and thus $\chi((\cL-\xi)\otimes \mathcal{O}_{\xi})=0$. An easy application of Riemann Roch and adjunction gives that $D\equiv (g_Y-1)\mathfrak{f}$ numerically. So we also recover \cite[Theo. 2.1]{aprodu-costa-miro}.
\end{rmk}

Concerning rank $r$ Ulrich bundles on projective bundles we have the following result.

\begin{thm}
\label{pullback}
  Let $(S,H)$ be a polarized surface with $H$ very ample and let $\cE$ be a rank two vector bundle on $S$ such that $\cE$ is (very) ample and spanned. Let $\cF$ be a rank $r$ vector bundle satisfying:
\begin{equation}
\begin{array}{ccc}
  H^i(S,\cF)=0 & \text{and} & H^i(S,\cF(-c_1(\cE)))=0,
\end{array}
\end{equation}
\noindent for $i=0,1,2$. Then on the projective bundle  $X:=\mathbb{P}(\cE)\xrightarrow{\pi} S$, the vector bundle $\cG:=\pi^*\cF\otimes \xi$ is Ulrich with respect to $\xi$.
\end{thm}
\begin{proof}
 $\cG$ will be Ulrich with respect to $\xi$ if the following vanishings are satisfied, for $i=0,1,2,3$:
\begin{equation}
 \begin{array}{l}
   1) \,\, 0=H^i(X,\pi^*\cF)\cong H^i(S,\cF), \\
   2) \,\,  0=H^i(X,-\xi\otimes\pi^*\cF), \\
   3) \,\, 0=H^i(X,-2\xi\otimes\pi^*\cF)\cong H^{3-i}(X,\pi^*\cF^{\vee}(c_1(\cE)+K_S))\cong H^{i-1}(S,\cF(-c_1(\cE))).

 \end{array}
\end{equation}
The vanishing in $2)$ are true since $\pi_{*}(-\xi)$ is the zero sheaf, so clearly the theorem holds.
\end{proof}

 \begin{rmk} We like to point out that Theorem \ref{pullback} is an generalization of \cite[Prop. 5, 2)]{beauville-ulrich-intro} to the case in which the base of the scroll is a surface and the  bundle on the base surface  with $H^{\bullet}(\cF)=0$  has rank $2$.
\end{rmk}

%

%
%
\section{Scroll over a $K3$ surface}
\label{Scroll over a $K3$ surface}

We start with the case of a $3$-dimensional scroll $X\subset \PP^5$ over a $K3$ surface $S$  that is a transversal intersection of the $8$-dimensional Grassmannian $\mathbb{G}(1, 5) \subset \PP^{14}$ with a $\mathbb{P}^8\subset \PP^{14}$.

 We consider first the case in which  $\mathrm{Pic}(S)$ has rank $1$. We will see that in this case the $3$-dimensional scroll   $X$ does not support any  Ulrich line bundle, while it carries  stable Ulrich bundles of rank $2$.

\begin{thm}
\label{Ulrich on general K3scroll}
  Let $X:=\PP(\cE)\subset\PP^5$ be  a $3$-dimensional scroll over a $K3$ surface $S$  such that $Pic(S)=\ZZ[H]$ with  $H$ a hyperplane section of $S$.
Then $X$ does not support  Ulrich line bundles. However,  $X$  supports stable  Ulrich bundles $\cG$ of rank $2$.
   \end{thm}
\begin{proof} By Corollary \ref{equivalent conditions for line bdl}, a line bundle $L=2\xi+D$ on $X$ is Ulrich if and only if
\begin{eqnarray*}
H^i(S, D)=0 \quad \text{for} \quad  i\ge 0  \quad \text{and}  \quad H^i(S, \cE(D))=0 \quad \text{for}\quad  i\ge 0.
\end{eqnarray*}

Similarly,  a line bundle $L=\xi+D$ on $X$ is  Ulrich if and only if
\begin{eqnarray*}
H^i(S, D)=0 \quad \text{for} \quad  i\ge 0  \quad \text{and}  \quad H^i(S, c_1(\cE)-D)=0 \quad \text{for}\quad  i\ge 0.
\end{eqnarray*}

Note that $\chi(S, D)=0$ implies that $D^2=-4$. But since  $\Pic(S)=\ZZ[H]$  there are no divisors $D$ with negative self-intersection.

Let us prove now that $X$ supports rank two Ulrich bundles. First of all, since $\Pic S= \ZZ[H]$ ,  by \cite{aprodu-farkas-ortega, faenzi}  $S$ supports special Ulrich bundles.
Let $\cF'$ be a special Ulrich bundle of rank $2$ on $S$. In particular, $c_1(\cF')=3H$. Then we can apply Theorem \ref{pullback} to $\cF:=\cF'(-H)$ (once it is observed that $c_1(\cF)=H$) to obtain rank two Ulrich bundle $\cG$ on $X$. Note that $\cG$ should be stable, since otherwise, by Proposition \ref{Ulrichproperties}, $\cG$ would be an extension of Ulrich line bundles contradicting the first part of the statement.
\end{proof}

We are now going to show the existence of a $3$-fold $X\subset\PP^5$ which is a scroll over a $K3$ surface $S$ with $Pic(S)=\ZZ[H]\oplus \ZZ[C]$ for some appropriate $H$ and $C$ such that $X$ carries both Ulrich line bundles and stable Ulrich bundles of rank two.

By \cite[Prop. 4.2, Theorem 4.6]{knutsen}, given integers $n\ge 2$, $d\ge 1$, $g \ge 0$ such that $d^2-4ng >0$ and $(d,g)\neq (2n+1,n+1)$, there exists a $K3$ surface $S$ of degree $2n$ in $\PP^{n+1}$ containing a smooth curve $C$ of degree $d$ and genus $g$. Furthermore, one can find such an $S$ with  $\mathrm{Pic}(S)=\ZZ[H]\oplus \ZZ[C]$ where $H$ is a hyperplane section of $S$, $H^2=2n$, $H\cdot C= d$, $C^2=2(g-1)$. We are interested in the case where $n=7$, $g=8$, $d=16$.

 \begin{lem}
 Let $S$ be a $K3$ surface as above having $\mathrm{Pic}(S)=\ZZ[H]\oplus \ZZ[C]$ with $H^2=C^2=14$ and $H\cdot C=16$. Then the following holds:
 \begin{itemize}
 \item[(i)] $S$ contains neither smooth rational curves nor smooth elliptic curves;
 \item[(ii)]  both $H$ and $C$ are indecomposable.
  \end{itemize}
 \end{lem}
 \begin{rmk}
 By \cite{SD}, item (i) yields that every primitive effective divisor on $S$ is very ample. By Lazarsfeld's Theorem \cite{La} along with item (ii), general curves in $|H|$ or $|C|$ are Brill-Noether-Petri general. In particular, their gonality is $k=5$, their Clifford dimension is $1$ and their Clifford index equals $k-2=3$.
 \end{rmk}
 \begin{proof}
 Let $D\equiv aH+bC$ with $a,b\in\ZZ$ be any divisor on $S$. Then $D^2=14a^2+14b^2+32ab=q(a,b)$ and (i) follows because neither $-2$ nor $0$ are represented by the quadratic form $q$ over the integers. Let $H\equiv N+M$ be any decomposition of $H$ into effectivedivisors. Then $14=N^2+M^2+2N\cdot M$. By (i), both $N$ and $M$ are ample and hence $N^2>0$, $M^2>0$ and $N\cdot M>0$. One gets a contradiction since the first positive even integer represented by $q$ is $6$ and the only possibility $N^2=M^2=6$, $N\cdot M=1$ is excluded by the Hodge Index Theorem. The same computation works for $C$.
 \end{proof}

 Since a general element $B\in |H|$ has no $g_7^2$, then it is a transversal linear section of the 8-dimensional Grassmannian $\mathbb{G}(1, 5) \subset \PP^{14}$ by \cite[Main Theorem A ]{mukai:curves-grassmannians}.

 Let $\zeta$ be a  $g_5^1$ on $B$  and consider the Lazarsfeld-Mukai bundle $E_{B, \zeta}$ on $S$ associated with the pair $(B, \zeta)$, sitting in the short exact sequence
  \begin{eqnarray}
  \label{LM_bdl on surf}
 0 \to H^{0}(B, \zeta)^{*} \otimes \cO_{S} \to  E_{B,\zeta} \to \omega_B\otimes \zeta^{-1}  \to 0.
  \end{eqnarray}
It is standard to show that the restriction $E_B$ of $E_{B,\zeta}$ to the curve $B$ sits in the short exact sequence
 \begin{eqnarray}
 \label{Mukaibdl on curve}
 0 \to \zeta \to E_B \to \omega_B\otimes \zeta^{-1}\to 0;
 \end{eqnarray}
such an extension is unique by \cite[\S 3]{mukai:curves-grassmannians}.

Standard computations show that $h^0(S,E_{B,\zeta})=6$.  Since  $\omega_B\otimes \zeta^{-1}$ is base point free, the  Lazarsfeld-Mukai bundles $E_{B,\zeta}$  is globally generated and thus provides a morphism $\varphi: S \to \GG(1,5)$.

\begin{claim}  $E_{B,\zeta}$ is very ample.
\end{claim}
\begin{proof}
It is enough to show that for any $p,q\in S$

 $$\di H^{0}(S, E_{B,\zeta} \otimes I_{p+q}) = \di H^{0}(S, E_{B,\zeta})-4=2.$$
 The sequence \eqref{LM_bdl on surf} remains exact when tensored by $I_{p+q}$, and thus we get
  \begin{eqnarray}
  \label{LM_bdl on surf}
 0 \to H^{0}(B, \zeta)^{*} \otimes I_{p+q} \to  E_{B,\zeta}\otimes I_{p+q}  \to \omega_B\otimes \zeta^{-1}\otimes I_{p+q}   \to 0
  \end{eqnarray}

Note that $ \di H^{0}(H^{0}(B, \zeta)^{*} \otimes I_{p+q})=0$ and $\di H^{0}(B,\omega_B\otimes \zeta^{-1}\otimes I_{p+q})= 2$ since $B$ does not have any   $g_7^2$. It follows that $\di H^{0}(S, E_{B,\zeta} \otimes I_{p+q})= 2$.
\end{proof}

The very ampleness of $E_{B,\zeta}$ implies that $\varphi$ is an embedding but is actually stronger. Geometrically it means that any two lines in $\PP^5$ (even those "infinitely close") do not meet and thus the variety $\PP( E_{B,\zeta})$ is embedded in $\PP^5$ by the tautological line bundle $\xi=\cO_{\PP( E_{B,\zeta})}(1)$.  Also note that  $E_{B,\zeta}= \varphi^{*} (\mathscr{S})$, where $\mathscr{S}$ is the rank $2$ universal bundle on $\GG( 1, 5)$.

We will now prove that on the variety $\PP( E_{B,\zeta})$ there is  a Ulrich line bundles of type (2) with respect to the tautological bundle $\xi$. We need the following:

\begin{lem}
In the above situation, assume that $B\in |H|$ is general. Then the Lazarsfeld-Mukai bundle $E_{B,\zeta}$ is both $\mu_H$-stable and $\mu_C$-stable.
\end{lem}
\begin{proof}
First of all, we prove that $E_{B,\zeta}$ is $\mu_H$-stable. By contradiction, assume the existence of a destabilizing sequence
\begin{equation}\label{destabilizing}
0\to M\to E_{B,\zeta}\to N\otimes I_\xi\to 0,
\end{equation}
where $M,N\in\mathrm{Pic}(S)$  satisfy $N\cdot H\leq\mu_H( E_{B,\zeta})=7\leq M\cdot H=(N-H)\cdot H$, and $\xi\subset S$ is a $0$-dimensional subscheme of length $l\geq 0$. Note that $5=c_2(E_{B,\zeta})=N\cdot M+l$. On the other hand, $N\cdot M\geq 5$ by \cite[Lem. 4.1]{LC} and hence $N\cdot (H-N)=N\cdot M=5$. Writing down $N\equiv aH+bC$, one checks that $N\cdot (H-N)$ is even and thus gets a contradiction.

Let us now show that $E_{B,\zeta}$ is $\mu_C$-stable, too. Again by contradiction, assume the existence of a destabilizing sequence with respect to $C$ of the same form as \eqref{destabilizing}, with $M$ and $N$ satisfying $N\cdot C\leq\mu_C( E_{B,\zeta})=8\leq M\cdot H=(N-H)\cdot C$. In particular, this implies that either $M\simeq N$ and $l=0$, or $\di\mathrm{Ext}^2(N\otimes I_\xi,M)=\di\mathrm{Hom}(M,N\otimes I_\xi)=0$. The former case can be excluded because $H\equiv M+N$ is primitive in $\mathrm{Pic}(S)$. The $\mu_H$-stability of $E_{B,\zeta}$ yields its simplicity along with the inequalities $N\cdot H>\mu_H( E_{B,\zeta})=7> M\cdot H=(N-H)\cdot H$, and thus $\mathrm{Hom}(N\otimes I_\xi,M)=0$. In particular, we get
\begin{eqnarray*}\di \mathrm{Ext}^1(N\otimes I_\xi,M)&=&-\chi(M^\vee\otimes N\otimes I_\xi)=-2-\frac{(N-M)^2}{2}+l=\\
&=&-2-(g-1)+2N\cdot M+l=-g-1+c_2(E_{B,\zeta})+N\cdot M=-4+N\cdot M.
\end{eqnarray*}

Let $\mathcal{P}_{N,l}$ be the parameter space for (simple) vector bundles $E$ sitting in a short exact sequence like \eqref{destabilizing} with $l(\xi)=l$. There is a natural surjective map $\mathcal{P}_{N,l}\to S^{[l]}$ with fiber over a $0$-dimensional scheme $\xi\in S^{[l]}$ given by $\mathbb{P}\mathrm{Ext}^1(N\otimes I_\xi, M)$; in particular, one has
$$
\dim \mathcal{P}_{N,l}=2l-5+N\cdot M=l.$$
Let $\mathcal{G}_{N,l}$ be the scheme parametrizing pairs $(B,\zeta)$ with $B\in |H|$ and $\zeta \in G^1_5(B)$ such that the Lazarsfeld-Mukai bundle $[E_{B,\zeta}]\in \mathcal{P}_{N,l}$. Since general curves in $|H|$ are Brill-Noether-Petri general, any component $\mathcal{G}$ of $\mathcal{G}_{N,l}$ dominating $|H|$ has dimension $$\di\mathcal{G}=\di |H|+\rho(8,1,5)=8.$$ On the other hand, $\mathcal{G}_{N,l}$ is birational to a Grassmann bundle over $\mathcal{P}_{N,l}$ with fibers isomorphic to $G(2,6)$ because the bundles in $\mathcal{P}_{N,l}$ are simple. Therefore, we get
$$\di\mathcal{G}=\di\mathcal{P}_{N,l}+\di G(2,6)=l+8.
$$
This implies $l=0$ and  $N\cdot (H-N)=5$, and thus again a contradiction.
\end{proof}

\begin{prop}
\label{Ulrich on K3scroll}
  Let $X:=\PP(\cE)\subset\PP^5$ be  a $3$-dimensional scroll over a $K3$ surface $S$  such that $Pic(S)=\ZZ[H]\oplus \ZZ[C]$ with  $H^2=C^2=14$ and $H\cdot C=16$.
Then $X$ supports no Ulrich line bundles of type $(1)$. On the other hand, $L:=2\xi+\pi^{*}D$ with $D\equiv H-C$ defines a Ulrich line bundle of type $(2)$ on $X$, and the same for its respective companion.   \end{prop}

\begin{proof}
By Corollary \ref{equivalent conditions for line bdl}, a line bundle $L=2\xi+D$ on $X$ is Ulrich if and only if
\begin{eqnarray*}
H^i(S, D)=0 \quad \text{for} \quad  i\ge 0  \quad \text{and}  \quad H^i(S, \cE(D))=0 \quad \text{for}\quad  i\ge 0.
\end{eqnarray*}
Note that $\chi(S, D)=0$ implies that $D^2=-4$. This along with $\chi(S, \cE(D))=0$ yields $H\cdot D=-2,$ and thus  $D\equiv H-C$.
We will first show that $H^{i}(S, D)=0$ for $i\ge 0$.

Note that $H^{2}(S, D)=H^{2}(S, H-C)\cong H^{0}(S, -H+C)=0$, as one can show by restricting to $C$. Hence, we get  $H^{1}(S, H-C)=H^{0}(S, H-C)=0$, where the last equality follows by restricting to $B\in |H|$.

It remains to check that $H^i(S, \cE(D))=0$ for $i\ge 0$, where $\cE=E_{B,\xi}$.
Since the bundle $\cE= E_{B,\zeta}$ is $\mu_H$-stable  and the slope of $\cE(C-2H)$ is negative it follows that $H^2(S, \cE (H-C))\cong H^0(S, \cE(C-2H)) = 0$. Since $\chi(\cE(H-C)=0$, it is enough to show that  $H^0(S, \cE(H-C)) = 0.$ Suppose there exists a non-zero section $s$ of $H^0(S, \cE(H-C))$. Note that $c_2(\cE(H-C)) = -1$ hence no section of $\cE(H-C)$ can have a pure codimension $2$ zero locus. In other words, the zero locus of $s$ should have a one-dimensional component. Therefore there exists an effective divisor $A$ (which we can assume to be prime) such that $H^0(S, \cE(H-C-A)) \neq  0$. The stability of $\cE$ yields $\mu_H (\cE(H-C-A))=\frac{1}{2}(2H-C-2A)\cdot H=6-A\cdot H>0$, that is, $1\leq A\cdot H\leq 5$. The Hodge Index Theorem thus implies $14A^2=A^2H^2\leq (A\cdot H)^2\leq 25$ and this is a contradiction because $S$ contains neither ($-2$)-curves nor smooth elliptic curves.

Let us now exclude the existence of Ulrich line bundles of type $(1)$. By Corollary \ref{equivalent conditions for line bdl}, a line bundle $L=\xi+D$ on $X$ is  Ulrich if and only if
\begin{eqnarray*}
H^i(S, D)=0 \quad \text{for} \quad  i\ge 0  \quad \text{and}  \quad H^i(S, c_1(\cE)-D)=0 \quad \text{for}\quad  i\ge 0.
\end{eqnarray*}
Note that  $\chi(S, D)=0$implies that $D^2=-4$. This along with $\chi(S, c_1(\cE)-D)=0$ yields $H\cdot D=7.$ Since
$D\equiv aH+bC$ for some $a,b \in \ZZ$, we get the equation
$14a+16b=7$ that has no integral solutions.
\end{proof}
\vspace{3mm}

Our next goal is to show the existence of some non-trivial extensions of the Ulrich line bundle constructed above with its Ulrich partner.

So set $A_1:=2\xi+H-C$ and let $A_2:=K_X+4\xi-A=C$ be its Ulrich partner. Then
$$
\Ext^1(A_2,A_1)\cong H^1(\PP(\cE),2\xi+H-2C)\cong H^1(S,S^2\cE(H-2C)).
$$
We compute $\chi(S^2\cE(H-2C))$. The Chern classes of $S^2\cE$ are as follows
\begin{eqnarray*}
 c_1(S^2\cE)=3c_1(\cE)=3H \quad \text{and} \quad  c_2(S^2\cE)=2c_1(\cE)^2+4c_2(\cE)=48.
\end{eqnarray*}
\noindent Therefore, we have
\begin{align}
    \nonumber    & c_1(S^2\cE(H-2C)=c_1(S^2\cE)+3(H-2C)=6H-6C   \\
    \nonumber & c_2(S^2\cE(H-2C))=c_2(S^2\cE)+2c_1(S^2\cE)(H-2C)+ 3(H-2C)^2=-42,
  \end{align}

\noindent and thus
$$
h^1(S^2\cE(H-2C))\geq -\chi(S^2\cE(H-2C))=-(\frac{1}{2}[c_1(S^2\cE(H-2C))^2-2c_2(S^2\cE(H-2C))]+3\chi(\cO_X))=24.
$$

\begin{prop}\label{cohgroups} In the previous setting, we have:
\begin{enumerate}
  \item $\Hom(A_2,A_1)=\Hom(A_1,A_2)=0$.
 \item $\Ext^2(A_2,A_1)=\Ext^3(A_2,A_1)=0$.
\end{enumerate}

\end{prop}
\begin{proof}
  \begin{enumerate}
    \item First of all, notice that $\Hom(A_2,A_1)\cong\HH^0(X,A_1-A_2)\cong \HH^0(S,S^2\cE(H-2C))$.  Now, since $\cE$ is $\mu_H$-stable, $S^2\cE$ is $\mu_H$-semistable. On the other hand,
    $$
    \mu_H(S^2\cE(H-2C))=\frac{1}{3}(6H-6C)H=-4<0.
    $$
    \noindent and therefore $S^2\cE(H-2C)$ has no global sections. Next,
    $$
    \Hom(A_1,A_2)\cong\HH^0(X,A_2-A_1)=\HH^0(X,2C-H-2\xi)=0.
    $$

 \item First of all, note that  $\Ext^3(A_2,A_1)\cong  \HH^3(S,S^2\cE(H-2C))=0$ trivially.  As concerns the other vanishing, we use the isomorphisms  $\Ext^2(A_2,A_1)\cong \HH^2(S,S^2\cE(H-2C))\cong \HH^0(S,S^2\cE^\vee(2C-H))$. The bundle $S^2\cE^\vee(2C-H)$ has no global sections since it is $\mu_C$-semistable and has positive $\mu_C$-slope.
  \end{enumerate}
\end{proof}

Proposition \ref{cohgroups} implies that we have a $23$-dimensional family of  non-trivial extensions
 \begin{equation}\label{extensionrankoneulrich}
   0\to A_1\to \cF\to A_2\to 0.
 \end{equation}
 By \cite[Lemma 4.2]{casanellas-hartshorne-geiss-schreyer} any non-trivial extension $\cF$ of $A_1$ by $A_2$ is simple, being $A_1$ and $A_2$ non-isomorphic Ulrich line bundles of the same slope. On the other hand,
$\cF$ (and $\cF(K_X)$), being Ulrich bundles, are $\mu$-semistable (or Gieseker semistable) with slopes satisfying
$$
\mu_{\xi}(\cF):=\frac{1}{2} c_1(\cF)\xi^2=16>\mu_{\xi}(\cF(K_X))=12.
$$
\noindent and therefore
$$
\Ext^3(\cF,\cF)\cong\Hom(\cF,\cF(K_X))\cong 0.
$$
 However, it does not seem to be easy to construct stable Ulrich bundles of rank $2$ on $X$ deforming such an $\cF$, since
$$ 14=-\chi(\cF\otimes \cF^\vee)\leq 23.$$
On the other hand, the existence of stable rank two bundles on $X$ again follows from Theorem \ref{pullback}.

\begin{thm}
\label{stable rk2 Ulrich on K3scroll}
  Let $X:=\PP(\cE)\subset\PP^5$ be  a $3$-dimensional scroll over a $K3$ surface $S$  such that $Pic(S)=\ZZ[H]\oplus \ZZ[C]$ with  $H^2=C^2=14$ and $H\cdot C=16$.
Then $X$ supports stable rank two Ulrich bundles $\cG$ with Chern classes $c_1(\cG)=2\xi+H$ and $c_2(\cG)=\xi^2+\xi H+11\mathfrak{f}$.
\end{thm}

\begin{proof}
As in the proof of Proposition \ref{Ulrich on general K3scroll}, thanks to the existence of special Ulrich bundles $\cF'$ of rank $2$ on $S$ (cf. \cite{faenzi}), $X$ supports rank two Ulrich bundles of the form $\cG:=\pi^*(\cF'(-H)))\otimes\xi$.  One easily verifies that their Chern classes are $c_1(\cG)=2\xi+H$ and $c_2(\cG)=\xi^2+\xi H+11\mathfrak{f}$, where $\mathfrak{f}$ is a fibre of the scroll $X$; in particular, one has $c_2(\cG)H=28$. In order to conclude that $\cG$ is stable, by Proposition \ref{Ulrichproperties} it is enough to remark that it cannot be an extension of the two Ulrich line bundles constructed in Proposition \ref{Ulrich on K3scroll} because otherwise it would satisfy  $c_2(\cG)H=32$.
\end{proof}

%
%

\section{Palatini Scroll}
\label{Palatini Scroll}
Let $X\subset\PP^5$ be the Palatini scroll (see \cite{ottaviani:scrolls}), that is, $X$ is the projectization of an Ulrich rank two bundle $\cE$ on the cubic surface $S\subset\PP^3$ with $c_1(\cE)=2H$ and $c_2(\cE)=5$.

The surface $S$ is the blow-up of $\PP^2$ at $6$ points in general position. If $\sigma$ is the blow-up map, we denote by $e_i$ for $i=1,\cdots, 6$ the exceptional curves and set $e_0:=\sigma^{*}{\mathcal O}_{\PP^{2}}(1)$.  The line bundle ${\mathcal O}_{S}(1)=\sigma^{*}{\mathcal O}_{\PP^{2}}(3)-\sum_{i=1}^6 e_i$ is the one giving the embedding of $S$ in $\PP^3$ and we denote by $H$ its class.

We are going to use Corollary (\ref{equivalent conditions for line bdl}) to prove the following result:

\begin{prop}
\label{Ulrich on Palatini}
  Let $X\subset\PP^5$ be the Palatini scroll.  There are no Ulrich line bundles of type $L=\xi+\pi^{*}D$. On the other hand, up to permutation of the exceptional divisors, there are three Ulrich line bundles of type $L=2\xi+\pi^{*}D$, jointly with their respective companions. They are:
\begin{equation}\label{listpalatini}
  \begin{array}{ll}
    L_1:=2\xi-e_0+e_1; & L_1':=4e_0-2e_1-e_2-e_3-e_4-e_5-e_6; \\
    L_2:=2\xi-2e_0+e_1+e_2+e_3+e_4 & L_2':=5e_0-2e_1-2e_2-2e_3-2e_4-e_5-e_6  \\
    L_3:=2\xi-3e_0+2e_1+e_2+e_3+e_4+e_5+e_6 & L_3':=6e_0-3e_1-2e_2-2e_3-2e_4-2e_5-2e_6.
  \end{array}
\end{equation}
  \end{prop}

\begin{proof}
If $\xi+D$ is an Ulrich line bundle, then the vanishings $0=\chi(D)=\chi(D-2H)$ yield $2DH=3$, hence a contradiction.

As concerns the second statement, let $L:=2\xi+D$ with $D:=ae_o+\sum_{i=1}^6 a_ie_i$ be an Ulrich line bundle. Applying Riemann Roch to the two equations $\chi(D)=\chi(\cE(D))=0$ we get
\begin{equation}\label{numerical}
\begin{array}{cc}
  a^2=\sum a_i^2 & \sum a_i=-3a-2.
\end{array}
\end{equation}
On the other hand, since the Ulrich companion $L'=H-D=(3-a)e_0-\sum (a_i+1)e_i$ should verify $H^0(L')=7$, we get $a<3$. From $H^0(ae_0+\sum_{i=1}^6a_ie_i)=0$ and equation (\ref{numerical}) we get that $a=2,1,0$ are not possible. Finally, again from (\ref{numerical}), we can exclude $a\leq -4$. Now, a one-by-one analysis of the remaining cases,
using the short exact sequence
$$
0\ra \cO_S(D)\ra\cE(D)\ra \cI_{Z|S}(2H+D)\ra 0,
$$
\noindent gives the result.

\end{proof}

We will use the previous Ulrich line bundles to construct rank two Ulrich bundles. In order to do this, take the Ulrich line bundles $D_1:=6e_0-3e_1-2e_2-2e_3-2e_4-2e_5-2e_6$ and $D_2:=5e_0-e_1-2e_2-2e_3-2e_4-2e_5-e_6$. One computes
$$
Ext^1(D_2,D_1)\cong H^1(e_0-2e_1-e_6)=\CC
$$
and thus get a non-trivial extension:
$$
0\ra \cO(D_1)\ra \cF\ra \cO(D_2)\ra 0.
$$
By \cite[Lemma 4.2]{casanellas-hartshorne-geiss-schreyer} any non-trivial extension $\cF$ of $\cO(D_1)$ by $\cO(D_2)$ is simple, being $\cO(D_1)$ and $\cO(D_2)$ non-isomorphic Ulrich line bundles of the same slope.

The former rank $2$ Ulrich bundles, being extensions of line bundles, cannot be stable. In order to construct stable rank two Ulrich bundles on the Palatini scroll we are going to use Theorem \ref{pullback} along with the following result:
\begin{prop}\label{rk2oncubicsurf}
  Let $S\subset\PP^3$ be a cubic surface. Then there exists a rank two bundle $\cF$ on $S$ with  Chern classes $c_1(\cF)=2\cL-H$, $c_2(\cF)=H^2-H\cdot \cL+\cL^2$  and satisfying $H^i(S,\cF)=H^i(S,\cF(-c_1(\cE))=0$ for $i=0,1,2$, where $\cL$ is any of the Ulrich line bundles on $S$ constructed in \cite{pons_llopis-tonini}.
\end{prop}
\begin{proof} We recall that $c_1(\cE)=\cO_{S}(2)$. The vanishings $H^i(S,\cF)=H^i(S,\cF(-c_1(\cE))=0$ for $i=0,1,2$ will be satisfied if we take $\cF:=\cF'\otimes \cO_{S}(-2)$ for some rank $2$ Ulrich vector bundle $\cF'$  with respect  to $\cO_{S}(2)$, assuming that such an $\cF$ exists.
By \cite{pons_llopis-tonini}, the polarized surface $(S,\cO_{S}(1))$ carries Ulrich line bundles and hence $(S,\cO_{S}(2))$ carries a Ulrich  bundle $\cF'$ of rank $2$ by \cite[Corollary in \S 2]{beauville-ulrich-intro}. Since $T_{\PP^{2}}$ is the only rank $2$ Ulrich bundle on $(\PP^{2},  \cO_{\PP^{2}}(2))$ (cf. \cite{coskun-genc:ulrich-veronese}), going through the proof  of  \cite[Corollary in \S 2]{beauville-ulrich-intro}, we get  $\cF'=p^*(T_{\PP^{2}}) + \cL$ where $p: S \to  \PP^{2}$ is a finite linear projection and $\cL$  is any of the Ulrich line bundle on $(S,\cO_{S}(1))$ provided in \cite{pons_llopis-tonini}. One easily checks that the Chern classes of $\cF=\cF'\otimes \cO_{S}(-2)$ are as in the statement.
\end{proof}

Now we can state:
\begin{thm}
\label{stable rk2 Ulrich on PalatiniScroll}
  On the Palatini scroll $X:=\PP(\cE)\subset\PP^5$ there are stable  rank two Ulrich bundles $\cG$ with Chern classes
  $c_1(\cG)=2\xi+\pi^{*}(2\cL-H)$ and $c_2(\cG)=\xi^2+\xi\cdot \pi^{*}(2\cL-H)+\pi^{*}(H^2-H\cdot \cL+\cL^2)$.
\end{thm}
\begin{proof}
  It is immediate, applying Theorem \ref{pullback} that the vector bundle $\cG:=\pi^*\cF\otimes \xi$ is Ulrich with respect to $\xi$, where  $\cF$ is the one constructed in  Proposition \ref{rk2oncubicsurf}. An easy Chern class computation  gives that $c_1(\cG)$ and $c_2(\cG)$ are as in the statement.

  A priori the vector bundle $\cG$ could be an extension of the line bundles given in Proposition \ref{Ulrich on Palatini}. For $\cL=e_0$, we get $\cF=\cF'\otimes \cO_{S}(-2)=p^*(T_{\PP^{2}}) + e_0-2H$ and $\cG:=\pi^*\cF\otimes\xi$ satisfies $c_1(\cG)=2\xi+\pi^{*}(-e_0+\sum_{i=1}^{6}e_i)$. On the other hand, no extension of the line bundles $L_i$ and $L_j'$ in (\ref{listpalatini}) realizes this first Chern class and thus the rank  two bundle $\cG$ is stable.
\end{proof}

%
%

\section{Scrolls over $\PP^2$, ${\mathcal Q}^2$, $\FF_1$}
\label{Scroll over P2;Q2;F1}
As it was pointed out in the introduction, $\PP^2$ is the base surface of two of the  smooth $3$-dimensional embedded in  $\PP^5$, namely, the Segre scroll and the Bordiga scroll.

 In codimension greater than $2$, from the list of variety of low degree, whose existence is known, we see that  the base surface  of the scroll is either $\PP^2$,  or a smooth quadric surface $\mathcal Q \subset \PP^3$,   or $\FF_1$ (cf., e.g.,  \cite{fania-livorni-deg10},  \cite[Table 1, Table 2]{besana-fania},  \cite{besana-biancofiore}, \cite[Remark 3.3 and \S 7]{alzati-besana}).

The goal of this section is to construct low rank Ulrich bundles over such three dimensional scrolls.

\begin{prop}\label{p_2}
Let $(X,L) = (\mathbb{P}(\cE), \xi)$ be a $3$-dimensional scroll over a surface $S$, with $S=\PP^2$, of
degree $3\le d\le 12$. Let  $\xi={\mathcal O}_{\mathbb{P}(\cE)}(1)$ be the tautological line bundle and $\pi: \mathbb{P}(\cE) \to S$ be the projection morphism. Let $X$ be embedded by $|L|=|\xi|$ in $\PP^N$. Then $(X,L) = (\mathbb{P}(\cE), \xi)$ does not support any Ulrich line bundle, unless
\begin{enumerate}
  \item either $c_1({\cE})={\mathcal O}_{\PP^2}(2)$ and $c_2({\cE})=1$;
  \item or $c_1({\cE})={\mathcal O}_{\PP^2}(4)$ and $c_2({\cE})=10$;
    \item or $c_1({\cE})={\mathcal O}_{\PP^2}(4)$ and $c_2({\cE})=6$;
  \item or $c_1({\cE})={\mathcal O}_{\PP^2}(5)$ and $c_2({\cE})=15$;
\end{enumerate}

\noindent In each case there exist exactly two Ulrich  line bundles, namely,
$$L_1:=2\xi+\pi^{*}{\mathcal O}_{\PP^2}(-1)  \quad \text{and its companion} \quad L_2=\pi^{*}({\mathcal O}_{\PP^2}(e-2)),$$ if $c_i(\cE)$ are as in $(2)$ and $(4)$,  where $e=c_1({\cE})$,  and
$$L_1:=2\xi+\pi^{*}{\mathcal O}_{\PP^2}(-2)  \quad \text{and its companion} \quad L_2=\pi^{*}({\mathcal O}_{\PP^2}(e-1))$$
if $c_i(\cE)$ are as in  $(1)$ and $(3)$.
\end{prop}
\begin{proof}
Let $L_1=a\xi+\pi^{*}{\mathcal O}_{\PP^2}(b)$  be a Ulrich line bundle on $\mathbb{P}(\cE)$, for some $a,b \in \ZZ$. Then $a=0,1,2$ by Corollary \ref{equivalent conditions for line bdl} .

 If $a=1$ then $L_1=\xi+\pi^{*}{\mathcal O}_{\PP^2}(b)$ and  $L_1$ Ulrich implies that  $H^i(\PP^2,{\mathcal O}_{\PP^2}(b))=H^i(\PP^2,{\mathcal O}_{\PP^2}(b)-c_1(\cE))=0$  for $i=0,1,2$ (Corollary \ref{equivalent conditions for line bdl}). Thus $\chi({\mathcal O}_{\PP^2}(b))=0$ which combined  with the Riemann-Roch Theorem  yields either $b=-1$ or $b=-2$. In both cases $\chi({\mathcal O}_{\PP^2}(b-e))\neq 0$ since $e=c_1({\cE})=5, 4, 2$.

If  $a=2$ then $L_1=2\xi+\pi^{*}{\mathcal O}_{\PP^2}(b)$ and  $L_1$ Ulrich implies that  $H^i(\PP^2,{\mathcal O}_{\PP^2}(b))=H^i(\PP^2,\cE(b))=0$  for $i=0,1,2$ (Corollary \ref{equivalent conditions for line bdl}).  As before $b=-1,-2$.  The vanishing $\chi(\cE(b))=0$ gives $2b^2+e^2+2be+6b+3e-2c_2+4=0$, where $e=c_1({\cE})$ and $c_2=c_2({\cE})$.

If $e=5$ and  $b=-2$, then $c_2=10$ and this is impossible. If   $e=5$ and  $b=-1$, then $c_2=15$ and the candidate Ulrich line bundle is $L_1=2\xi+\pi^{*}{\mathcal O}_{\PP^2}(-1)$.

 In this case $X\subset  \PP^6$ has degree $d=e^2-c_2=10$ and $X$ is a linear determinantal variety, namely, the degeneracy locus of a generic vector bundle homomorphism $u:{\mathcal O}_{\PP^6}^{\oplus 3} \to
{\mathcal O}_{\PP^6}(1)^{\oplus 5}$.  One easily verifies (cf., e.g.,  \cite{faenzi-fania}) that the vector bundle $\cE$  is the cokernel of
\begin{eqnarray}
\label{E}
0\to {\mathcal O}_{\PP^2}(-1)^{\oplus 5} \to {\mathcal O}_{\PP^2}^{\oplus 7} \to \cE \to 0.
\end{eqnarray}

Let us now check that the line bundle $L_1=2\xi+\pi^{*}{\mathcal O}_{\PP^2}(-1)$ satisifies
$H^i(\PP^2, {\mathcal O}_{\PP^2}(-1))=H^i(\PP^2,\cE(-1))=0$ for $i=0,1,2$. Since $H^i(\PP^2, {\mathcal O}_{\PP^2}(-1))=0$ for $i=0,1,2$ it remains to show that $H^i(\PP^2,\cE(-1))=0$ for $i=0,1,2$ and this follows from (\ref{E}).

If  $e=4$ and  $b=-1$, then $c_2=10$. In this case the $3$-fold $X \subset \PP^5$ has degree $6$ and  the candidate Ulrich line bundle is $L_1=2\xi+\pi^{*}{\mathcal O}_{\PP^2}(-1)$.   Let us now check that the line bundle $L_1=2\xi+\pi^{*}{\mathcal O}_{\PP^2}(-1)$ satisifies $H^i(\PP^2, {\mathcal O}_{\PP^2}(-1))=H^i(\PP^2,\cE(-1))=0$ for $i=0,1,2$. Since $H^i(\PP^2, {\mathcal O}_{\PP^2}(-1))=0$ for $i=0,1,2$ it remains to show that $H^i(\PP^2,\cE(-1))=0$ for $i=0,1,2$.  The Bordiga scroll  is a linear determinantal variety,\cite[\S 3]{ottaviani:scrolls}, namely, the degeneracy locus of a generic vector bundle homomorphism ${\mathcal O}_{\PP^5}(-1)^{\oplus 4} \to
{\mathcal O}_{\PP^5}^{\oplus 3}$.  One easily verifies (cf., e.g.,  \cite{faenzi-fania}) that the vector bundle $\cE$  is the cokernel of
\begin{eqnarray}
\label{E!}
0\to {\mathcal O}_{\PP^2}(-1)^{\oplus 4} \to {\mathcal O}_{\PP^2}^{\oplus 6} \to \cE \to 0.
\end{eqnarray}
from which it follows that $H^i(\PP^2,\cE(-1))=0$ for $i=0,1,2$.

If $e=4$ and  $b=-2$, then $c_2=6$. In  this case the candidate Ulrich line bundle with respect to the polarization $\xi$   is $L_1=2\xi+\pi^{*}{\mathcal O}_{\PP^2}(-2)$. The condition $H^i(\mathcal{O}_{\PP^2}(-2))=0$ for all $i\geq 0$ is trivially satisfied. Furthermore, the vector bundle $\cE$ is stable (cf. \cite[Prop 1.3]{ionescu3}) and from \cite[1.2.5 Lemma]{okonek-schneider-spindler} it follows that $H^{0}(\cE(-2))=0$. On the other hand it is easy to verify that $H^2(\cE(-2))=0$ and that $\chi(\cE(-2))=0$ and thus $H^{i}(\cE(-2))=0$ for all $i\ge 0$. In this case the companion of $L_1$ is  $L_2=\pi^{*}{\mathcal O}_{\PP^2}(3)$.

 If  $e=2$ and  $b=-1$, then $c_2=3$ and this is impossible. If $e=2$ and  $b=-2$,  then $c_2=1$. In this case the candidate Ulrich line bundle is $L_1=2\xi+\pi^{*}{\mathcal O}_{\PP^2}(-2)$. Moreover the $3$-fold $X \subset \PP^5$ is the Segre variety $\PP^1\times \PP^2$, its degree is $3$  and $\cE\cong \cO_{\PP^2}(1)^2$.  An easy computiation gives that $H^i(\PP^2, {\mathcal O}_{\PP^2}(-2))=H^i(\PP^2,\cE(-2))=0$ for $i=0,1,2$.
   \end{proof}
   \begin{rmk}
The existence of Ulrich line bundles in  the cases $(1)$, $(2)$ and $(4)$ of Proposition \ref{p_2}  (i) was already known using another approach. Precisely:
  \begin{enumerate}
    \item[$\bullet$] \emph{(Case $c_1(\cE)=2,c_2(\cE)=1$ on $\PP^2$)}. Then $\cE\cong \cO_{\PP^2}(1)^2$ and $X\cong \PP^2\times\PP^1\subset\PP^5$.  The Segre threefold  was known to support two Ulrich line bundles: $\cO_X(H_1)=\cO_X(1,0)$ and $\cO_X(2H_2)=\cO(0,2)$ in the usual base (see \cite[Prop. 3.2]{costa-miro_roig-pons_llopis}). As we have
        $$
        \begin{array}{lc}
          \xi=H_1+H_2 & H=H_1,
        \end{array}
        $$
        \noindent we recover the two Ulrich line bundles provided by Corollary(\ref{equivalent conditions for line bdl}), namely, $2\xi-2H=2H_2$ and $H=H_1$.
     \vspace{3mm}

        In the next two cases  $X$ is a linear determinantal varieties (cf. \cite{ottaviani:scrolls} and \cite{fania-livorni-deg10}) and thus the existence of Ulrich line bundles follows from \cite{miro_roig-pons_llopis:surfaces_P4}.

    \item[$\bullet$] \emph{(Case $c_1(\cE)=4,c_2(\cE)=10$ on $\PP^2$)}.  In this case $X\subset\PP^5$ is the Bordiga scroll, it has degree $6$ and resolution:
        $$
        0\ra\cO_{\PP^5}(-4)^3\ra\cO_{\PP^5}(-3)^4\ra\cI_X\ra 0.
        $$
        which corresponds to the following $\cO_{\PP^2}$-resolution of $\cE$:
        $$
        0\ra\cO_{\PP^2}(-1)^4\ra\cO_{\PP^2}^6\ra\cE\ra 0.
        $$
        \noindent Hence, the line bundles $L_1:=2\xi-H=-K_X$ and $L_2:=2H$ are Ulrich.
    \item[$\bullet$] \emph{(Case $c_1(\cE)=5$ and $c_2(\cE)=15$ on $\PP^2$)}.  The bundle $\cE$ has resolution
  $$
  0\ra\cO_{\PP^2}(-1)^5\ra\cO_{\PP^2}^7\ra\cE\ra 0.
  $$
  \noindent Hence, $X:=\mathbb{P}(\cE)\subset\PP^6$ is a linear determinantal variety of degree $10$ (as in  \cite{fania-livorni-deg10}) with resolution
  $$
  0\ra\cO_{\PP^6}(-5)^6\ra\cO_{\PP^6}(-4)^{15}\ra\cO_{\PP^6}(-3)^{10}\ra\cI_X\ra 0.
  $$
  Then $L_1:=2\xi-H=-K_X+H$ and $L_2:=3H$ are Ulrich line bundles on $X$.
  \end{enumerate}
\end{rmk}

%
%

\begin{prop}\label{Q_2_and_F1}
Let $(X,L) = (\mathbb{P}(\cE), \xi)$ be a $3$-dimensional scroll over a surface $S$, with $S$ either  ${\mathcal Q}^2$ or $\FF_1$
of degree $8\le d\le 11$. Let  $\xi={\mathcal O}_{\mathbb{P}(\cE)}(1)$ be the tautological line bundle and $\pi: \mathbb{P}(\cE) \to S$ be the projection morphism. Let $X$ be embedded by $|L|=|\xi|$ in $\PP^N$. Then $(X,L) = (\mathbb{P}(\cE), \xi)$ does not support any Ulrich line bundle, unless   $S={\mathcal Q}^2$,  $c_1({\cE})={\mathcal O}_{{\mathcal Q}^2}(3,3)$, $c_2({\cE})=9$ and there exist exactly two Ulrich  line bundles, namely,

\begin{center}
$L_1:=\xi+\pi^{*}{\mathcal O}_{{\mathcal Q}^2}(-1,2)\quad \text{and its companion} \quad L_2=\xi+\pi^{*}{\mathcal O}_{{\mathcal Q}^2}(2,-1).$
\end{center}
\end{prop}
\begin{proof}  Assume  that the base surface is ${\mathcal Q}^2$ and let  $L_1=a\xi+\pi^{*}{\mathcal O}_{{\mathcal Q}^2}(\alpha, \beta)$, for some $\alpha,\beta \in \ZZ$  be a Ulrich line bundle on $\mathbb{P}(\cE)$. Corollary \ref{equivalent conditions for line bdl}  gives $a=0,1,2$

If $a=1$ then $L_1=\xi+\pi^{*}{\mathcal O}_{{\mathcal Q}^2}(\alpha,\beta)$ and  $L_1$ Ulrich implies that  $H^i({\mathcal Q}^2, {\mathcal O}_{{\mathcal Q}^2}(\alpha,\beta))=H^i({\mathcal Q}^2,{\mathcal O}_{{\mathcal Q}^2}(\alpha,\beta)-c_1(\cE))=0$  for $i=0,1,2$ (Corollary  \ref{equivalent conditions for line bdl}). Thus $\chi({\mathcal O}_{{\mathcal Q}^2}(\alpha,\beta))=\chi({\mathcal O}_{{\mathcal Q}^2}(\alpha-3,\beta-3))$ which combined  with the Riemann-Roch Theorem gives  $\alpha \beta+\alpha+ \beta+1=(\alpha+1)(\beta+1)=0$ and  $\alpha \beta-2\alpha-2 \beta+4=(\alpha-2)(\beta-2)=0$, and thus either $\alpha=-1$ and $\beta=2$, or $\alpha=2$ and $\beta=-1$.  We need only to verify that $$
H^{i}({\mathcal Q}^2, {\mathcal O}_{{\mathcal Q}^2}(-1,2))=H^{i}({\mathcal Q}^2, {\mathcal O}_{{\mathcal Q}^2}(-1,2)-c_1(\cE))=0 \quad  \text{for}\quad  i\ge 0.$$
By the K\"unneth formula, this holds true precisely when $c_1({\cE})={\mathcal O}_{{\mathcal Q}^2}(3,3)$.

If $a=2$ then $L_1=2\xi+\pi^{*}{\mathcal O}_{\mathcal Q}^2(\alpha,\beta)$ and  $L_1$ Ulrich implies that  $H^i({\mathcal Q}^2, {\mathcal O}_{{\mathcal Q}^2}(\alpha,\beta))=H^i({\mathcal Q}^2,\cE({\mathcal O}_{{\mathcal Q}^2}(\alpha,\beta)))=0$  for $i=0,1,2$ (Corollary  \ref{equivalent conditions for line bdl}). Thus $\chi({\mathcal O}_{{\mathcal Q}^2}(\alpha,\beta))=\chi(\cE({\mathcal O}_{{\mathcal Q}^2}(\alpha,\beta)))=0$.  From $\chi({\mathcal O}_{{\mathcal Q}^2}(\alpha,\beta))=0$ it follows that $\alpha \beta+\alpha+ \beta+1=0$, that is $(\alpha+1)(\beta+1)=0$ and from $\chi(\cE({\mathcal O}_{{\mathcal Q}^2}(\alpha,\beta)))=0$, combined  with the Riemann-Roch Theorem it follows that  $2\alpha \beta+5\alpha+5 \beta+8=0$. One can easily see that the only possibility is  $\alpha= \beta=-1$.  From \cite[Remark 7.5]{fania-livorni-deg10} the possible presentations of $\cE$ are
 $$0 \to {\mathcal O}_{{\mathcal Q}^2}(0,-3)\to  {\mathcal O}_{{\mathcal Q}^2}(1,-1) \oplus {\mathcal O}_{{\mathcal Q}^2}(1,0)\oplus {\mathcal O}_{{\mathcal Q}^2}(1,1)\to \cE \to 0$$
 and
 $$0 \to {\mathcal O}_{{\mathcal Q}^2}(-3,0)\to  {\mathcal O}_{{\mathcal Q}^2}(-1,1) \oplus {\mathcal O}_{{\mathcal Q}^2}(0,1)\oplus {\mathcal O}_{{\mathcal Q}^2}(1,1)\to \cE\to 0,$$
 from which it follows that $H^{0}(\cE({\mathcal O}_{{\mathcal Q}^2}(-1,-1)))\neq 0$. Thus this case cannot occur.
 \vspace{2mm}

Assume now that $S=\FF_1$, and let $\pi:\FF_1\to\mathbb{P}^1$ be the natural projection map. Denote by $C_0$ and $f$ the unique section of self-intersection $-1$ and the class of a fiber of $\pi$, respectively. If for some $a,\alpha,\beta \in \ZZ$ the line bundle $L_1=a\xi+\pi^{*}(\alpha C_0+ \beta f)$ on $\mathbb{P}(\cE)$ is Ulrich, then $a=0,1,2$ by Corollary \ref{equivalent conditions for line bdl}.

If $a=1$ then $L_1=\xi+\pi^{*}(\alpha C_0+ \beta f)$ and  $L_1$ Ulrich implies that  $H^i(\FF_1, \alpha C_0+ \beta f)=H^i(\FF_1,\alpha C_0+ \beta f -c_1(\cE))=0$  for $i=0,1,2$ (Corollary  \ref{equivalent conditions for line bdl}). Thus $\chi(\alpha C_0+ \beta f))=\chi((\alpha-3) C_0+ (\beta -5)f))=0$, because  $c_1({\cE})=3C_0+5f$. The Riemann-Roch Theorem thus yields the equations $\alpha^2-\alpha-2\beta-2\alpha \beta-2=(\alpha+1)(\alpha-2\beta-2)=0$ and $\alpha^2+3\alpha+4\beta-2\alpha \beta-10=(\alpha-2)(\alpha-2\beta+5)=0$, the only integral solutions of which are $(\alpha, \beta)=(2, 0), (-1,2)$.  If $(\alpha, \beta)=(2, 0)$ then the cohomology groups $H^i(\FF_1, \alpha C_0+ \beta f)=H^i(\FF_1, 2C_0)$  are not all zero. Analogously, for $(\alpha, \beta)=(-1, 2)$ not all the cohomology groups   $H^i(\FF_1, (\alpha-3) C_0+ (\beta-5) f)=H^i(\FF_1, -4C_0-3f)$  vanish.

If $a=2$ then $L_1=a\xi+\pi^{*}(\alpha C_0+ \beta f)$, and  $L_1$ Ulrich implies that  $H^i(\FF_1, \alpha C_0+ \beta f)=H^i(\FF_1,\cE(\alpha C_0+ \beta f))=0$  for $i=0,1,2$ (Corollary  \ref{equivalent conditions for line bdl}). Thus $\chi(\alpha C_0+ \beta f))=(\alpha+1)(\alpha-2\beta-2)=0$ and $\chi(\cE(\alpha C_0+ \beta f))=0$, that is, $\alpha^2-3\alpha-5\beta-2\alpha \beta-9=0$  if $c_2(\cE)=10$, and  $\alpha^2-3\alpha-5\beta-2\alpha \beta-8=0$  if $c_2(\cE)=11$. These two set of equations do not have any integral solution.

Hence we conclude that there are no Ulrich line bundles on the $3$-dimensional scroll $X = \mathbb{P}(\cE)$ over  $\FF_1$ with  $c_1({\cE})=3C_0+5f$ and  $c_2({\cE})=10, \text{or} \, \,11$.
\end{proof}

As concerns rank two Ulrich bundles, their  existence  has been proved in \cite{costa-miro_roig-pons_llopis} for the Segre scroll and in \cite{miro_roig-pons_llopis:surfaces_P4} for linear determinantal varieties. It remains to consider the $3$-fold scroll over $\PP^2$ with  $c_1(\cE)={\mathcal O}_{\PP^2}(4)$ and $c_2({\cE})=6$ , the $3$-fold scroll over ${\mathcal Q}^2$ with $c_1({\cE})={\mathcal O}_{{\mathcal Q}^2}(3,3)$ and  $c_2({\cE})=9$ and the $3$-fold scrolls over $\FF_1$ with $c_1({\cE})=3C_0+5f$ and  $c_2({\cE})=10, \text{or} \, \,11$.

We will now investigate the remaining cases.  It will turn out that all these cases support stable Ulrich bundles of rank $2$.

Let us start with the case of   a $3$-dimensional scroll $X$ over $\PP^2$.   We will construct stable rank two Ulrich bundles on $X$ using  Theorem \ref{pullback} along with the following result

%


\begin{prop}\label{rk2onP2}
  For $e\geq 1$, there exists a rank two bundle $\cF$ on $\PP^2$ with  Chern classes $c_1(\cF)= \cO_{\PP^{2}}(e-3)$, $c_2(\cF)=\frac{e^2-3e+4}{2}$  and satisfying $H^i(\PP^2,\cF)=H^i(\PP^2,\cF(-e))=0$ for $i=0,1,2$.
\end{prop}
\begin{proof} The vanishings $H^i(\PP^{2},\cF)=H^i(\PP^{2},\cF(-e))=0$ for $i=0,1,2$ will be satisfied if we take $\cF:=\cF'\otimes \cO_{\PP^{2}}(-e)$ for some rank $2$ Ulrich vector bundle $\cF'$  with respect  to $\cO_{\PP^{2}}(e)$. Such an $\cF'$ exists by \cite[Proposition 4]{beauville-ulrich-intro}. From \cite[Proposition 2.1]{casnati-ulrich}  we easily see that the Chern classes of $\cF=\cF'\otimes  \cO_{\PP^{2}}(-e)$ are as in the statement.
\end{proof}

Now we can state:
\begin{thm}
\label{stable rk2 Ulrich on Scrolls overP2}
  On a  scroll $X:=\PP_{\PP^{2}}(\cE)$ there are stable  rank two Ulrich bundles $\cG$ with Chern classes
  $c_1(\cG)=2\xi+\pi^{*}(\cO_{\PP^{2}}(e-3))$ and $c_2(\cG)=\xi^2+\xi\cdot \pi^{*}(\cO_{\PP^{2}}(e-3))+\frac{e^2-3e+4}{2}\mathfrak{f}$.
\end{thm}
\begin{proof}
  Applying Theorem \ref{pullback}, it is immediate that the vector bundle $\cG:=\pi^*\cF\otimes \xi$ is Ulrich with respect to $\xi$, where  $\cF$ is the one constructed in  Proposition \ref{rk2onP2} with $e=c_1(\cE)$. An easy Chern class computation  gives that $c_1(\cG)$ and $c_2(\cG)$ are as in the statement.

  A priori the vector bundle $\cG$ could be an extension of the Ulrich line bundles $L_1$ and $L_2$  given in Proposition \ref{p_2}.  No extensions of the Ulrich line bundles $L_1$ and $L_2$  have second Chern class as that of $\cG$ and thus the rank  two bundle $\cG$ is stable.
\end{proof}

Let us move on to the case of $(X,L) = (\mathbb{P}(\cE), \xi)$  a $3$-dimensional scroll over ${\mathcal Q}^2$,  with $c_1({\cE})={\mathcal O}_{{\mathcal Q}^2}(3,3)$, $c_2({\cE})=9$.  In order to show the existence of stable Ulrich bundles of rank two on $X$ we follow the method proposed in \cite{casanellas-hartshorne-geiss-schreyer}: we will first compute the dimension $l$ of simple rank two Ulrich bundles $\cF$ on $X$ obtained as an extension of the Ulrich line bundles from \ref{Q_2_and_F1}. Then we will show the existence of a modular family of simple rank two Ulrich bundles and show that the dimension of the modular family at $\cF$ is higher than $l$. Then we can conclude that the generic element of this modular family should be stable. The existence of such a modular family on $X$ is guaranteed by the following result (see \cite[Prop. 2.10]{casanellas-hartshorne-geiss-schreyer}):

\begin{prop}\label{modular}
On a nonsingular projective variety X, any bounded family of simple bundles $\cF$ with given rank and Chern classes satisfying $H^2(\cF\otimes\cF^{\vee})=0$ has
a smooth modular family.
\end{prop}

\begin{prop}\label{rk2Ulrich on scroll over Q}
Let $(X,L) = (\mathbb{P}(\cE), \xi)$ be a $3$-dimensional scroll over ${\mathcal Q}^2$,  with $c_1({\cE})={\mathcal O}_{{\mathcal Q}^2}(3,3)$, $c_2({\cE})=9$. Let  $\xi={\mathcal O}_{\mathbb{P}(\cE)}(1)$ be the tautological line bundle and $\pi: \mathbb{P}(\cE) \to {\mathcal Q}^2$ be the projection morphism. Let $X$ be embedded by $|\xi|$ in $\PP^N$. Then there exists a family of dimension $7$  of simple rank $2$ Ulrich bundles $\cF$ on $X$. Moreover, $H^2(\cF\otimes\cF^{\vee})=0$ and $\chi(\cF\otimes\cF^{\vee})=-14$.
\end{prop}
\begin{proof}

From Proposition \ref{Q_2_and_F1} it follows that  $L_1:=\xi+\pi^{*}{\mathcal O}_{{\mathcal Q}^2}(-1,2)$ and $L_2=\xi+\pi^{*}{\mathcal O}_{{\mathcal Q}^2}(2,-1)$  are the only two Ulrich line bundles on  $\mathbb{P}(\cE)$ with respect to $\xi$.  Note that
$$
\dim \Ext^1(L_1,L_2)=h^1({\mathcal Q}^2, \pi^{*}{\mathcal O}_{{\mathcal Q}^2}(3,-3))= h^1(\PP^1,{\mathcal O}_{\PP^1}(-3))\cdot  h^0(\PP^1,{\mathcal O}_{\PP^1}(3))=8$$
So we get  a family of dimension $7$  of non trivial-extensions $\cF$ of $L_1$ by $L_2$
\begin{equation}\label{aaa}
  0\ra \xi+\pi^{*}{\mathcal O}_{{\mathcal Q}^2}(-1,2)\ra \cF\ra \xi+\pi^{*}{\mathcal O}_{{\mathcal Q}^2}(2,-1) \ra 0.
\end{equation}

 By \cite[Lemma 4.2]{casanellas-hartshorne-geiss-schreyer} any non-trivial extension $\cF$ of $L_1$ by $L_2$ is simple, being $L_1$ and $L_2$ non-isomorphic Ulrich line bundles of the same slope.  Finally, a standard computation using (\ref{aaa}) shows that $H^2(\cF\otimes\cF^{\vee})=0$ and $\chi(\cF\otimes\cF^{\vee})=-14$.
\end{proof}

\begin{thm}\label{stablerank2UlrichonQ2}
  Let $(X,L) = (\mathbb{P}(\cE), \xi)$ be a $3$-dimensional scroll over ${\mathcal Q}^2$,  with $c_1({\cE})={\mathcal O}_{{\mathcal Q}^2}(3,3)$, $c_2({\cE})=9$. Then $X$ supports stable rank 2 Ulrich bundles $\cF$ with Chern classes $c_1(\cF)= 2\xi +\pi^*\cO_{\mathcal{Q}^2}(1,1)$ and $c_2(\cF)=\xi^2+\xi\cdot\pi^*\cO_{\mathcal{Q}^2}(1,1)+5\mathfrak{f}$.
\end{thm}
\begin{proof}
  Let $\mathcal{M}$ the family of rank $2$ simple Ulrich bundles on $X$ with Chern classes as in the statement and satisfying $H^2(\cF\otimes\cF^{\vee})=0$. Because Ulrich bundles are semistable, this family is bounded. Moreover, by Proposition \ref{rk2Ulrich on scroll over Q}, it is non-empty. Therefore, by Proposition \ref{modular}, $\mathcal{M}$ has a smooth modular family whose dimension at a point $[\cF]$ can be computed as $Ext^1(\cF,\cF)=1-\chi(\cF\otimes\cF^{\vee})=15$. Since the family of rank two Ulrich bundles that have a presentation as an extension of Ulrich line bundles has dimension seven, we can conclude that the generic element of the modular family is stable.
\end{proof}

Finally we deal with the case in which the base surface of the scroll is $\FF^1$.

\begin{thm}\label{rk2Ulrich on scroll over F1}
Let $(X,L) = (\mathbb{P}(\cE), \xi)$ be a $3$-dimensional scroll over $\FF_1$,  with $c_1(\cE)=3C_0+5f$, $c_2(\cE)=10, \,  {\text or} \, 11$. Let  $\xi={\mathcal O}_{\mathbb{P}(\cE)}(1)$ be the tautological line bundle and $\pi: \mathbb{P}(\cE) \to \FF_1$ be the projection morphism. Let $X$ be embedded by $|\xi|$ in $\PP^N$. Then there are stable  rank two Ulrich bundles $\cG$ with Chern classes
  $c_1(\cG)=2\xi+\pi^{*}(C_0+2f)$ and $c_2(\cG)=\xi^2+\xi \cdot \pi^{*}(C_0+2f)+6\mathfrak{f}$.
\end{thm}
\begin{proof}
  By Theorem \ref{pullback}, the vector bundle $\cG:=\pi^*\cF\otimes \xi$ is Ulrich with respect to $\xi$ as soon as  $\cF:=\cF'\otimes \cO_{\FF_1}(-c_1(\cE))$ for some rank two Ulrich vector bundle $\cF'$ on $\FF_1$ with respect  to $ \cO_{\FF_1}(c_1(\cE))$. Such an $\cF'$ exists by  \cite[Theorem 3.4]{aprodu-costa-miro}  and is special;  in particular, Proposition \ref{ulrichchernclassessurface} yields $c_1(\cF')=7C_0+12f$ and $c_2(\cF')=35$. An easy Chern class computation  gives that $c_1(\cG)$ and $c_2(\cG)$ are as in the statement.

The rank  two bundle $\cG$ is stable since there are no Ulrich line bundle on  the $3$-dimensional scroll over $\FF_1$.
\end{proof}

%
%

\section{pushforwards}\label{ultima}
In the previous sections we were concerned with the study of Ulrich bundles on projective scrolls that can be constructed as a (modified) pull-back of a vector bundle on the base variety. In this section we are going to illustrate a method to perform the opposite operation, namely, starting with a Ulrich bundle satisfying a certain property we will obtain
an Ulrich bundle on the base $S$ of the same rank. Let us recall that a general hyperplane section $\tilde{S}$ of $X$ has the structure of a blow-up of $S$ at $c_2(\cE)$ points. We consider the following diagram:

\begin{equation}
    \label{composition}
    \xymatrix@-2ex{
    \tilde{S} \, \, \ar@{^{(}->}[r]^i \ar[rd]_{\pi'} & X \ar[d]^{\pi} \\
                                      & S,
    }
\end{equation}
where $i$ is the inclusion and $\pi'$ is the blow-up map; we denote by $E_i$ the exceptional divisors of the latter.

\begin{thm}\label{pushforwardulrich}
  Let $\pi:X:=\PP(\cE)\rightarrow S$ be a projective bundle threefold over a surface $S$ and let $\cG$ be an Ulrich bundle on $X$ with respect to the tautological line bundle $\xi$ of rank $r$. Let us suppose that on the generic fibre $F=\pi(s)$, $s\in S$, the vector bundle $\cG$ splits as follows: $\cG_F\cong\cO_{\PP^1}(1)^{ r}$. Then $\pi_*(\cG\otimes i_*(\cO_{\tilde{S}}(\sum E_i))$ is a rank $r$ Ulrich vector bundle on $S$ with respect to $c_1(\cE)$.
\end{thm}

\begin{proof}
 Let  $\tilde{S}\subset X$ be a generic section of $|\xi|$ and let us call $\tilde{H}$ the very ample line bundle on $\tilde{S}$ obtained as a restriction of $\xi$ to $\tilde{S}$. From adjunction theory, we know that $(S, c_1(\cE))$ is the reduction of  $(\tilde{S},\tilde{H})$ and thus
\begin{equation}\label{adjunction}
K_{\tilde{S}}+\tilde{H}\cong \pi'^*(K_S+c_1(\cE)).
\end{equation}

Now, applying the hypothesis concerning the generic splitting type of $\cG$, we are in position to apply \cite[Theorem 4.2]{casnati-kim} to conclude that $\pi'_*((i^*\cG)\otimes\cO_{\tilde{S}}(\sum E_i))$ is Ulrich of rank $r$ with respect to $c_1(\cE)$. It only remains to apply the projection formula
$$
\pi'_*((i^*\cG)\otimes\cO_{\tilde{S}}(\sum E_i))\cong (\pi\circ i)_*((i^*\cG)\otimes\cO_{\tilde{S}}(\sum E_i))\cong\pi_*(\cG\otimes i_*(\cO_{\tilde{S}}(\sum E_i)))
$$
\noindent to obtain the statement of the Theorem.
\end{proof}
Beauville pointed out  in \cite[Corollary 1.]{beauville-ulrich-intro} that if $X$ is a $n$-dimensional projective variety carrying an Ulrich bundle of rank $r$ with respect to a certain very ample line bundle $\cO_X(H)$ then it also supports rank $rn!$ Ulrich bundles with respect to  $\cO_X(dH)$, $d\geq 2$. The previous Theorem could be interpreted as a potential method to construct on $S$ Ulrich bundles of \emph{low rank} with respect to \emph{high degree} polarizations of the form $c_1(\cE)=dH$. It should be underlined, however, that if we start with some Ulrich bundle $\cF$ on $S$ with respect to the polarization $c_1(\cE)$ and apply first Theorem \ref{pullback} to $\cF(-c_1(\cE))$ and then Theorem \ref{pushforwardulrich} to the resulting vector bundle on $X$, we just recover the original Ulrich bundle $\cF$. More specifically we have the following Proposition.

\begin{prop}\label{bijection}
  Let $\pi:X:=\PP(\cE)\rightarrow S$ be a projective bundle threefold over a surface $S$. Then there exists a bijection:

$$
\left\{ \begin{array}{c}
          \text{Ulrich bundles $\cF$ of rank $r$ on $S$ } \\
          \text{with respect to $c_1(\cE)$}
        \end{array}\right\}_{\biggm/ \cong_{iso}}\Leftrightarrow\left\{\begin{array}{c}
          \text{Ulrich bundles $\cG$ of rank $r$ on $X$} \\
          \text{ with respect to $\xi$ such that} \\
           \text{$\cG_{|\pi^{-1}(s)}\cong\cO_{\PP^1}(1)^r$ for $s\in S$}
        \end{array}\right\}_{\biggm/ \cong_{iso}}
$$
given by the maps
\begin{equation*}
  \phi:\cF\quad\mapsto\quad \cG:=\pi^*\cF(-c_1(\cE))\otimes\xi;
\end{equation*}
\noindent and
  \begin{equation*}
  \psi:\cG\quad\mapsto\quad \cF:=\pi_*(\cG\otimes i_*(\cO_{\tilde{S}}(\sum E_i)).
\end{equation*}
\end{prop}

\begin{proof}

  It is immediate to see that $(\pi^*\cF(-c_1(\cE))\otimes\xi)_{| \pi^{-1}(s)}\cong\cO_{\PP^1}(1)^r$ for any $s\in S$. This, together with Theorems \ref{pullback} and \ref{pushforwardulrich} show that $\phi$ and $\psi$ are well-defined. To conclude, we are going to show that both compositions define isomorphisms.\\
\vspace{0.1cm}
\emph{Claim: $\psi\circ\phi$ is an isomorphism:}\\
This follows from the following chain of isomorphisms:
$$
\begin{array}{cccl}
  \pi_*((\pi^*\cF(-c_1(\cE))\otimes\xi)\otimes i_*\cO_{\tilde{S}}(\sum E_i))& \stackrel{Proj. form.}{\cong} & \pi_*\circ i_*(i^*(\pi^*\cF(-c_1(\cE))\otimes\xi)\otimes \cO_{\tilde{S}}(\sum E_i)) & \cong\\
   \pi'_*(\pi'^*(\cF)\otimes\cO_{\tilde{S}}(\tilde{H}+\sum E_i+\pi'^*(-c_1(\cE)))& \stackrel{\tilde{H}+\cO_{\tilde{S}}(\sum E_i)\equiv \pi'^*(c_1(\cE))}{\cong} & \pi'_*\circ\pi'^*(\cF) &\cong\\
  \cF & & &
\end{array}
$$

\noindent  where the last isomorphism is obtained using the projection formula and the fact that $\pi'_*(\cO_{\tilde{S}})\cong \cO_S.$\\

\vspace{0.1cm}
 \emph{Claim: $\phi\circ\psi$ is an isomorphism:}\\
 Take $\cG$ a rank $r$ Ulrich bundle on $X$ with respect to $\xi$ such that $\cG_{|\pi^{-1}(s)}\cong\cO_{\PP^1}(1)^r$ for $s\in S$. Following the same proof as in the previous Claim, it follows that $\cG$ and $\phi\circ\psi(\cG)$ restricted to $\tilde{S}$ are isomorphic:

 \[\xymatrix{
 i^*\circ\pi^*(\pi_*(\cG\circ i_*(\sum E_i))\otimes\cO_S(-c_1(\cE)))\otimes\cO_{\tilde{S}}(\tilde{H})
               \ar@{}[r]|*=0[@]{\cong}
               \ar^-{Proj. form.}@{}[r]|*=0[@]{\cong}\quad
               &\pi'^*(\pi'_*(i^*\cG\otimes\cO_{\tilde{S}}(\sum E_i))\otimes \cO_S(-c_1(\cE)) )\otimes\cO_{\tilde{S}}(\tilde{H})\\
            \qquad \qquad  \qquad \qquad \qquad \qquad
               \ar^-{\text{\cite[Cor. 2.3]{casnati-kim}}}@{}[r]|*=0[@]{\cong}
              & i^*\cG\otimes\cO_{\tilde{S}}(\sum E_i)\otimes \cO_{\tilde{S}}(\pi'^*(-c_1(\cE))\otimes\cO_{\tilde{S}}(\tilde{H})\\
                 \ar@{}[r]|*=0[@]{\cong}
        \ar^-{\tilde{H}+\cO_{\tilde{S}}(\sum E_i)\equiv \pi'^*(c_1(\cE))}@{}[r]|*=0[@]{\cong}
& i^*\cG }\]


\noindent Therefore, it remains to show that this isomorphism can be extended to the entire vector bundle. In order to see that, notice that
 for any two vector bundles $\cA$ and $\cB$ on $X$ we have the short exact sequence:
 $$
 0\longrightarrow\cA^{\vee}\otimes\cB\otimes\xi^{\vee}\longrightarrow\cA^{\vee}\otimes\cB\longrightarrow\cA^{\vee}\otimes\cB_{|\tilde{S}}\longrightarrow 0.
 $$
 \noindent If we compute the long exact sequence of cohomology groups with $\cA:=\cG$ and $\cB:=\phi\circ\psi(\cG)$ we obtain that
 $$
 Hom(\cG,\phi\circ\psi(\cG))\cong H^0(\cG^{\vee}\otimes(\phi\circ\psi(\cG)))\cong H^0(\cG^{\vee}\otimes\phi\circ\psi(\cG)_{\tilde{S}})\cong Hom(\cG_{\tilde{S}},(\phi\circ\psi(\cG))_{\tilde{S}}),
 $$
 \noindent since $H^0(\cG^{\vee}\otimes(\phi\circ\psi(\cG)(-\xi)))=0$ (because $\cG$ and $\phi\circ\psi(\cG)$ are semistable of the same slopes) and $H^1(\cG^{\vee}\otimes(\phi\circ\psi(\cG)(-\xi)))=0$ by  Proposition \ref{properties-taut-bundle}. This allows us to conclude.
 \end{proof}

 \begin{rmk}
 The Proposition \ref{bijection}, jointly with Theorem \ref{stablerank2UlrichonQ2}, can be used to give a proof of the existence of stable rank 2 Ulrich bundles on the quadric $\mathcal{Q}^2$ with respect to $\cO_{{\mathcal Q}^2}(3,3)$.  An alternative approach is offered in \cite{antonelli}.
\end{rmk}

%
%

\bibliographystyle{alpha}

\begin{thebibliography}{CMRPL12}

\bibitem[AB10]{alzati-besana}
A.~Alzati and G.M. Besana.
\newblock Criteria for very-ampleness of rank-two vector bundles over ruled
  surfaces.
\newblock {\em Canad.J.Math.}, 62(6):1201--1227, 2010.

\bibitem[ACMR18]{aprodu-costa-miro}
Marian Aprodu, Laura Costa, and Rosa~Maria Mir\'{o}-Roig.
\newblock Ulrich bundles on ruled surfaces.
\newblock {\em J. Pure Appl. Algebra}, 222:1:131--138, 2018.

\bibitem[AFO17]{aprodu-farkas-ortega}
Marian {Aprodu}, Gavril {Farkas}, and Angela {Ortega}.
\newblock {Minimal resolutions, Chow forms and Ulrich bundles on K3 surfaces}.
\newblock {\em J. Reine Angew. Math.}, 730:225--249, 2017.

\bibitem[{A}nt18]{antonelli}
{V}incenzo {A}ntonelli.
\newblock Characterization of {U}lrich bundles on {H}irzebruch surfaces.
\newblock {\em ArXiv e-print math.AG/1806.10380v1}, 2018.

\bibitem[BB05]{besana-biancofiore}
A.~Biancofiore and G.M. Besana.
\newblock Degree eleven projective manifolds of dimension greater than or equal
  to three.
\newblock {\em Forum Math.}, 17(5):711--733, 2005.

\bibitem[Bea16]{beauv-abelian-ulrich}
Arnaud Beauville.
\newblock Ulrich bundles on abelian surfaces.
\newblock {\em Proc. Amer. Math. Soc.}, 144(11):4609--4611, 2016.

\bibitem[Bea18]{beauville-ulrich-intro}
Arnaud Beauville.
\newblock An introduction to {U}lrich bundles.
\newblock {\em Eur. J. Math.}, 4(1):26--36, 2018.

\bibitem[BF05]{besana-fania}
G.M. Besana and M.L. Fania.
\newblock The dimension of the {H}ilbert scheme of special threefolds.
\newblock {\em Communications in Algebra}, 33(10):3811--3829, 2005.

\bibitem[BGS87]{buchweitz-greuel-schreyer}
Ragnar-Olaf Buchweitz, Gert-Martin Greuel, and Frank-Olaf Schreyer.
\newblock Cohen-{M}acaulay modules on hypersurface singularities, {II}.
\newblock {\em Invent. Math.}, 88(1):165--182, 1987.

\bibitem[Cas17]{casnati-ulrich}
Gianfranco Casnati.
\newblock Special {U}lrich bundles on non-special surfaces with {$p_g=q=0$}.
\newblock {\em Internat. J. Math.}, 28(8):1750061, 18, 2017.

\bibitem[CCH{\etalchar{+}}17]{coskun-costa-hui-miro}
Izzet Coskun, Laura Costa, Jack Huizega, Rosa~Maria Mir\'{o}-Roig, and Matthew
  Woolf.
\newblock {U}lrich {S}chur bundles on flag varieties.
\newblock {\em J. Algebra}, 474:49--96, 2017.

\bibitem[CG17]{coskun-genc:ulrich-veronese}
Emre Coskun and Ozhan Genc.
\newblock {U}lrich bundles on {V}eronese surfaces.
\newblock {\em Proc. Amer. Math. Soc.}, 145:4687--4701, 2017.

\bibitem[CHGS12]{casanellas-hartshorne-geiss-schreyer}
Marta Casanellas, Robin Hartshorne, Florian Geiss, and Frank-Olaf Schreyer.
\newblock Stable {U}lrich bundles.
\newblock {\em Internat. J. Math.}, 23(8):1250083, 50, 2012.

\bibitem[CK17]{casnati-kim}
G.~Casnati and Y.~Kim.
\newblock Ulrich bundles on blowing up (and an erratum).
\newblock {\em C.R.Acad.Sci.Paris, Ser I}, 355(3):1291--1297, 2017.

\bibitem[CMR16]{costa-miro-ulrich-grass}
Laura Costa and Rosa~Maria Mir\'{o}-Roig.
\newblock Homogeneous {ACM} bundles on a {G}rassmannian.
\newblock {\em Adv. Math}, 289:95--113, 2016.

\bibitem[CMRPL12]{costa-miro_roig-pons_llopis}
Laura Costa, Rosa~Maria Mir{\'o}-Roig, and Joan Pons-Llopis.
\newblock The representation type of {S}egre varieties.
\newblock {\em Adv. Math.}, 230(4-6):1995--2013, 2012.

\bibitem[ESW03]{eisenbud-schreyer-weyman}
David Eisenbud, Frank-Olaf Schreyer, and Jerzy Weyman.
\newblock Resultants and {C}how forms via exterior syzygies.
\newblock {\em J. Amer. Math. Soc.}, 16(3):537--579, 2003.

\bibitem[{F}ae18]{faenzi}
{D}aniele {F}aenzi.
\newblock Ulrich bundles on {K}3 surfaces.
\newblock {\em ArXiv e-print math.AG/1807.07826}, 2018.

\bibitem[FF14]{faenzi-fania}
Daniele Faenzi and Maria~Lucia Fania.
\newblock On the {H}ilbert scheme of varieties defined by maximal minors.
\newblock {\em Math. Res. Lett.}, 21:297--311, 2014.
\newblock 10.4310/MRL.2014.v21.n2.a8.

\bibitem[FL97]{fania-livorni-deg10}
M.L. Fania and E.L. Livorni.
\newblock Degree ten manifolds of dimension n greater than or equal to 3.
\newblock {\em Math. Nachr.}, 188:79--108, 1997.
\newblock 10.1002/mana.19941690111.

\bibitem[Har77]{hartshorne:ag}
Robin Hartshorne.
\newblock {\em Algebraic {G}eometry}.
\newblock Springer-Verlag, New York, 1977.
\newblock Graduate Texts in Mathematics, No. 52.

\bibitem[Ion90]{ionescu3}
Paltin Ionescu.
\newblock Embedded projective varieties of small invariants. iii.
\newblock In {\em Algebraic geometry ({L'A}quila, 1988)}, volume 1417 of {\em
  Lecture Notes in Math.}, pages 138--154. Springer-Verlag, Berlin, 1990.

\bibitem[Knu02]{knutsen}
A.~L. Knutsen.
\newblock Smooth curves on projective {K}3 surfaces.
\newblock {\em Math. Scand.}, 90(2):215--231, 2002.

\bibitem[Laz86]{La}
Robert Lazarsfeld.
\newblock Brill-{N}oether-{P}etri without degenerations.
\newblock {\em J. Differential Geometry}, 23(3):299--307, 1986.

\bibitem[LC13]{LC}
Margherita Lelli-Chiesa.
\newblock Stability of rank-3 {L}azarsfeld-{M}ukai bundles on {K}3 surfaces.
\newblock {\em Proc. Lond. Math. Soc.}, 107(2):451--479, 2013.

\bibitem[MRPL13]{miro_roig-pons_llopis:surfaces_P4}
Rosa~M. Mir{\'o}-Roig and Joan Pons-Llopis.
\newblock Representation {T}ype of {R}ational {ACM} {S}urfaces
  {$X\subseteq\Bbb{P}^4$}.
\newblock {\em Algebr. Represent. Theory}, 16(4):1135--1157, 2013.

\bibitem[MRPL14]{miro-roig-pons-llopis:fano}
Rosa~M. Mir{\'o}-Roig and Joan Pons-Llopis.
\newblock {$n$}-dimensional {F}ano varieties of wild representation type.
\newblock {\em J. Pure Appl. Algebra}, 218(10):1867--1884, 2014.

\bibitem[Muk93]{mukai:curves-grassmannians}
Shigeru Mukai.
\newblock Curves and {G}rassmannians.
\newblock In {\em Algebraic {G}eometry and related topics (Inchon, 1992)},
  Conf. Proc. Lecture Notes Algebraic Geom., I, pages 19--40. Int. Press,
  Cambridge, MA, 1993.

\bibitem[OSS80]{okonek-schneider-spindler}
Christian Okonek, Michael Schneider, and Heinz Spindler.
\newblock {\em Vector bundles on complex projective spaces}, volume~3 of {\em
  Progress in Mathematics}.
\newblock Birkh\"auser Boston, Mass., 1980.

\bibitem[Ott92]{ottaviani:scrolls}
Giorgio Ottaviani.
\newblock On {$3$}-folds in {$\bold P\sp 5$} which are scrolls.
\newblock {\em Ann. Scuola Norm. Sup. Pisa Cl. Sci. (4)}, 19(3):451--471, 1992.

\bibitem[PLT09]{pons_llopis-tonini}
Joan Pons-Llopis and Fabio Tonini.
\newblock A{CM} bundles on del {P}ezzo surfaces.
\newblock {\em Matematiche (Catania)}, 64(2):177--211, 2009.

\bibitem[SD74]{SD}
B.~Saint-Donat.
\newblock Projective models of {K}3 surfaces.
\newblock {\em Amer. J. Math.}, 96:602--639, 1974.

\bibitem[Ulr84]{ulrich:high-numbers}
Bernd Ulrich.
\newblock Gorenstein rings and modules with high numbers of generators.
\newblock {\em Math. Z.}, 188(1):23--32, 1984.

\end{thebibliography}

\newcommand{\etalchar}[1]{$^{#1}$}
\def\cprime{$'$} \def\cprime{$'$} \def\cprime{$'$} \def\cprime{$'$}
  \def\cprime{$'$} \def\cprime{$'$} \def\cprime{$'$} \def\cprime{$'$}
  \def\cprime{$'$}

\vfill
\end{document}